\documentclass[11pt]{amsart}

\usepackage{geometry}                
\usepackage{graphicx}
\usepackage{amssymb}
\usepackage{epstopdf}
\usepackage{hyperref}
\hypersetup{
	linktoc=true
}

\newcommand{\R}{\mathbb{R}}

\newcommand{\Z}{\mathbb{Z}}

\newcommand{\T}{\mathbb{T}}

\newcommand{\HH}{\mathbb{H}}
\newcommand{\LL}{\mathbb{L}}
\newcommand{\Lp}{\mathbb{L}^p}

\newtheorem{theorem}{Theorem}
\newtheorem{lemma}{Lemma}
\newtheorem{proposition}{Proposition}
\newtheorem{corollary}{Corollary}

\newtheorem{question}{Question}

\theoremstyle{definition}
\newtheorem{definition}{Definition}
\theoremstyle{remark}
\newtheorem{remark}{Remark}

\newtheorem{hypothesis}{Hypothesis}

\begin{document}

\title{On the quantitative quasi-isometry problem: transport of Poincar\'e inequalities and different types of quasi-isometric distortion growth}
\author{Vladimir Shchur}
\address{Universit\'e Paris-Sud, F-91405 Orsay Cedex, France}
\address{Current address: Wellcome Trust Sanger Institute, CB10 1SA Hinxton, Cambridgeshire, UK}
\address{vlshchur@gmail.com}
\address{vs3@sanger.co.uk}

\begin{abstract}
We consider a quantitative form of the quasi-isometry problem. We discuss several arguments which lead us to different results and bounds of quasi-isometric distortion: comparison of volumes, connectivity etc. Then we study the transport of Poincar\'e constants by quasi-isometries and we give sharp lower and upper bounds for the homotopy distortion growth for an interesting class of hyperbolic metric spaces.
\end{abstract}

\keywords{Quasi-isometries, hyperbolic spaces, Poincar\'e constants, quantitative quasi-isometry problem}
\subjclass[2010]{51F99}

\maketitle





\markright{On the quantitative quasi-isometry problem}

\section{Introduction}

In this article we shall study a quantitative form of the quasi-isometry problem: we will give lower and upper bounds for quasi-isometry constants $\lambda$ and $c$ for different classes of spaces. Along the way, we will give a method to transport Poincar\'e inequalities by quasi-isometries that leads to sharp bounds for certain spaces.

The quantitative quasi-isometry problem consists in evaluating how close two metric spaces can be at various scales, see \cite{shchur-t}. Specifically, let $E, F$ be two metric spaces. Consider a ball of radius $R$ in the first space $E$ and take a $(\lambda,c)$-quasi-isometric embedding of this ball in $F$. We are interested in the behaviour of the infimum of the sum $\lambda+c$ of quasi-isometry constants as a function of $R$.

\subsection{First examples}

Since one may always take $\lambda=1$ and $c=R$, the deviation between any two spaces is at most linear.

Volume considerations show that any space with polynomial volume growth deviates linearly from any space of exponential volume growth (see Proposition \ref{ExpPoly}).

Connectedness considerations provide a lower bound of $\sqrt{R}$ for the embeddings of Euclidean or hyperbolic balls to trees. In the hyperbolic case, this is sharp (see Proposition \ref{BigTree}).

This suggests that, in the family of Gromov hyperbolic metric spaces, deviations should be of the order $\sqrt{R}$. Indeed, we show (see Proposition \ref{denseSet}) that given two thick enough hyperbolic metric spaces, one can map a $\sqrt{R}$-dense subset of an $R$-ball of the first space into the second one with $\sqrt{R}$ distortion. However, we have been unable to extend such embeddings to the full $R$-ball. 

There seems to be a rather subtle obstruction to doing this. For instance, we show in Proposition \ref{linearTree} that mapping a tree into hyperbolic space requires linear distortion. This is based on the notion of separation, cf. \cite{BS}, \cite{BST}.

\subsection{Main result}

Our main result is another step towards capturing such obstructions. We shall consider a class of negatively curved locally homogeneous Riemannian manifolds which are not simply connected, but nevertheless hyperbolic. We prove a sharp linear lower bound on the distortion of embeddings which are homotopy equivalences.

Let $\T^n$ denote the $n$-dimensional torus. Given positive numbers $\mu_1 \le\cdots\le\mu_n$, denote by $Z_{\mu}=\T^n\times\R$, where the product space is equipped with the Riemannian metric $dt^2+\sum_i e^{2\mu_it}dx_i^2$. The universal cover of $Z_{\mu}$ is a Riemannian homogeneous space. $Z_{\mu}$ is a hyperbolic metric space. Its ideal boundary is a product of circles, each of which has a metric which is a power of the usual metric. Thus $Z_{\mu}$ can be viewed as a hyperbolic cone over this fractal torus. Essentially our theorem states that the quasi-isometric distortion growth function between such spaces is linear if one requires maps to be isomorphic on fundamental groups.

\begin{theorem}
\label{lower}
(Rough version. For a precise statement, see Theorem \ref{lowerBound}). Every $(\lambda,c)$-quasi-isometric embedding of an $R$-ball of $Z_{\mu}$ into $Z_{\mu'}$ which is a homotopy equivalence satisfies
$$\lambda+c \ge const\left(\frac{\sum\mu_i}{\mu_n} - \frac{\sum\mu_n'}{\mu_n'}\right)R.$$
\end{theorem}
Conversely, there exist homotopy equivalences with linearly growing distortion, 
$$\lambda+c\leq const\max|\mu_i-\mu'_i|R,$$ 
from an $R$-ball of $Z_{\mu}$ into $Z_{\mu'}$. This is a special case of a more general result which we describe next.

\bigskip

In a hyperbolic metric space, we give a formula for the distance in terms of the visual distance on the ideal boundary. Using this formula we find quasi-isometry constants for the restriction on balls of a map $\Theta$ between $X$ and $Y$ which is a kind of radial extension of a homeomorphism $\theta$ between ideal boundaries. The following is a non technical statement of Theorem \ref{Theta},
see Section \ref{adequi}
for a complete statement.

\begin{theorem}
\label{upper}
Let $X$, $Y$ be hyperbolic metric spaces. Let $\theta:\partial X\to\partial Y$ be a homeomorphism. We define the following function. For $R>0$,
$$K(R) = \sup\left\{\left|\log\frac{d_{y_0}(\theta(\xi_1), \theta(\xi_2))}{d_{x_0}(\xi_1, \xi_2)}\right|| d_{x_0}(\xi_1, \xi_2) \ge e^{-R} \vee d_{y_0}(\theta(\xi_1), \theta(\xi_2)) \ge e^{-R}\right\}.$$
Here $d_{x_0}$, $d_{y_0}$ denote visual metrics on ideal boundaries.
Then there exists a $(K(R), K(R))$-quasi-isometry between $B_X(x_0,R)$ and $B_Y(y_0,R)$. 
\end{theorem}

For spaces $Z_\mu$, we show that $K(R) = \max_i|\mu_i/\mu_i' - 1|R$. Then we give an example of a pair of non-quasi-isometric negatively curved locally homogeneous manifolds and a homeomorphism $\theta$ between their ideal boundaries with $K(R) \lesssim \log R$. This shows that subpolynomial (possibly logarithmic) distortion growths also occur in the world of hyperbolic metric spaces.

\subsection{Proof}

The proof of Theorem \ref{lower} involves several results which could have an independent interest and more applications. First, we study the transport of Poincar\'e inequalities by quasi-isometries. For this purpose we will use kernels to regularize transported functions. Kernels allow us to transport functions from $Y$ to $X$ while controlling quantitatively their Poincar\'e constants. 

Now we give more details on the proof of the theorem itself. It has several steps. First we introduce non-trivial double-covering spaces $\tilde Z$ and $\tilde Z'$ of $Z=Z_\mu$ and $Z'=Z_{\mu'}$. We prove that $\Theta$ lifts to a $(\lambda_1,2c_1)$-coarse Lipschitz map. Then we take the test-function $e^{\pi i x_n}$ on $\tilde Z'$ which depends only on one coordinate $x_n$. It varies very slowly outside of some ball, so the absolute value of the transported and regularised function $v$ on $\tilde Z$ stays close to $1$. Lemmas \ref{grad2cocycle} and \ref{transpCocycles} allow us to control how the lower bound of Poincar\'e constant changes under transport. This helps us get a lower bound for the Poincar\'e constant of $\tilde Z$ in terms of $\{\mu_i\}, \{\mu_i'\}$ and the constants of quasi-isometric embedding. We also prove an upper bound for the Poincar\'e constant of $\tilde Z$ in Theorem \ref{expPoincare}. The combination of these results provides a lower bound for the homotopy distortion growth for $Z$ and $Z'$.

\section{Basic definitions}

\begin{definition}\label{defQIS}
Two metric spaces $X$ and $Y$ are said to be roughly quasi-isometric if there exists a pair of maps $f: X \to Y$, $g: Y \to X$ and two constants $\lambda > 0$ and $c \ge 0$ such that
\begin{itemize}
\item $|f(x) - f(y)| \le \lambda |x - y| + c$ for every $x, y \in X$,
\item $|g(x') - g(y')| \le \lambda |x' - y'| + c$ for every $x', y' \in Y$,
\item $|g(f(x)) - x| \le c$ for every $x \in X$,
\item $|f(g(x')) - x'| \le c$ for every $x' \in Y$.
\end{itemize}
The word \emph{rough} is often dropped away.
\end{definition}

The first two conditions mean that $f$ and $g$ are nearly Lipschitz if we are looking from afar. The two latter conditions provide that $f$ and $g$ are nearly inverse of each other. It is easy to check that the composition of two quasi-isometries is also a quasi-isometry. So, quasi-isometries provide an equivalence relation on the class of metric spaces.

\begin{remark}
Definition \ref{defQIS} is invariant under taking inverse maps.
\end{remark}

\begin{definition}
A map $f:E\to F$ between metric spaces is a {\em rough $(\lambda_1, c_1, \lambda_2, c_2)$-quasi-isometric embedding} if for any two points $x, y$ of $E$
$$\frac{1}{\lambda_2}(|x - y|_E - c_2) \le |f(x) - f(y)|_F \le \lambda_1 |x - y|_E + c_1.$$
\end{definition}

This definition includes quasi-isometries (with $\lambda_1=\lambda_2$ and $c_1 = c_2$) but it does not require the existence of a nearly inverse map. We introduced four constants instead of two because for our quantitative questions we would like to follow what is the role of each inequality in this definition.

\medskip

We introduce the following definition to formalize our quantitative problem.
\begin{definition}\label{DG}
Let $X, Y$ be metric spaces, $x_0, y_0$ their base points respectively. The quasi-isometric distortion growth is the function
\begin{eqnarray*}
D_{G}(X,x_0,Y,y_0)(R) = \inf\{d|\exists f:B_X(x_0,R) \to Y \text{ a $(\lambda_f,c_f)$-quasi-isometric embedding}
\\\text{such that }
f(x_0)=y_0 \text{ and } d=\lambda_f+c_f \}.
\end{eqnarray*}
\end{definition}
We will study the growth of $D_G$ as a function of $R$.

\section{General discussion}

Here we collect elementary arguments which provide lower bounds on quasi-isometry constants.

\subsection{Comparison of volumes}\label{volumes}

First we will show that comparison of volumes in the domain and in the range plays an important role. 

By volume of a subset in a metric space, we mean the number of balls of a fixed radius needed to cover that subset.



Consider a space $X$ with an exponential volume growth (for example, hyperbolic plane $\HH^2$) and a space $Y$ with a polynomial volume growth (for example, euclidean space $\R^n$), then quasi-isometry constants between balls $B_R(X)$ and $B_R(Y)$ grow linearly in $R$: $\lambda_R+c_R = \Omega(R)$.

\begin{proposition}\label{ExpPoly}
Let $X$ be a space with exponential volume growth and $Y$ be a space with polynomial volume growth. Then for any $(\lambda,c)$-quasi-isometric embedding of a ball $B_X(R)$ into $Y$ we have $c \ge const\cdot R$.
\end{proposition}

For the sake of simplicity, in the proof, we will assume that the volume of a ball $B_X(R)$ in $X$ is $e^R$ and the volume of a ball $B_Y(R)$ in $Y$ is $R^\alpha$.

\begin{proof}
Let $B_X(R)$ be a ball in $X$, $f:B_X(R)\to Y$ be a $(\lambda,c)$-quasi-isometric embedding. Then the diameter of the image $f(B_X(R))$ is $\le 2\lambda R + c$. Consider a maximal set $S$ of points in $B_X(R)$ such that pairwise distances between these points are at least $2c$. We can estimate the cardinality of $S$ as $\#(S)\sim Vol(B_X(R)/Vol(B_X(2c))$. For any two points $s_1, s_2 \in S$ the distance between their images is at least $c/\lambda$. Hence, the volume of $f(B_X(R))$ is at least $\#(S)\times Vol(B_Y(c/\lambda))$.

So, on the one hand $Vol(f(B_X(R))) \le Vol(B_Y(2\lambda R + c))$ and on the other hand $Vol(f(B_X(R))) \ge Vol(B_Y(c/\lambda))Vol(B_X(R)/Vol(B_X(2c))$. We get
$$(c/\lambda)^{\alpha} e^{R-2c}\le (2\lambda R+c)^\alpha.$$

For $R$ big enough, this inequality can be satisfied only if exponential term disappears, that is $c = R/2$.

\end{proof}

\begin{remark}
The same argument yields lower bounds on quasi-isometry constants between balls of the same radius in spaces of different exponential growths. This does not prevent such spaces from being quasi-isometric. For instance, \cite{Papasoglu} shows that two regular trees of degrees at least $4$ are always quasi-isometric. 
The quasi-isometry provided by \cite{Papasoglu} does not preserve the distance to a fixed point.
\end{remark}

\subsection{Connectedness}

Another property which can detect a difference in the coarse geometry of two spaces is connectedness. For example if we cut a ball from a tree then it will fall into several components, but this does not happen with hyperbolic plane. First, we define coarse connectivity.

\begin{definition}
A map $f: X\to Y$ between two metric spaces is called $c$-connected if for any point $x\in X$ and any real number $\delta > 0$ there exists $\varepsilon > 0$ such that if a point $x' \in X$ satisfies $d(x, x') < \varepsilon$ then $d(f(x), f(x')) < c+ \delta$,
\end{definition}

\begin{definition}
1. A metric space $X$ is called $c$-connected if for any two open sets $U, V \subset X$ such that $X = U \cup V$, the intersection of a $c$-neighbourhood of $U$ and $V$ is not empty: $(U+c)\cap V \neq \emptyset$.

2. Equivalently, a metric space $X$ is $c$-connected if for any two points $x, x' \in X$ there exists a $c$-connected map $f: [0,1]\to X$ such that $f(0) = x$ and $f(1) = x'$.
\end{definition}
First and second definitions are evidently equivalent.

Now we are ready to illustrate our idea. In the following proposition we can take for example hyperbolic plane as the space $X$.

\begin{proposition}\label{connected}
Let $X$ be a geodesic metric  space. We suppose that for any points $x, y$ and any positive real numbers $R$ and $R' \le R/2$ the set $B_x(R)\setminus B_y(R')$ is connected and non-empty. Let $Y$ be a tree, let $f: B_x(R)\to Y$ be a $(\lambda_1, \lambda_2, c_1, c_2)$-quasi-isometric embedding. Then $R \le 12\lambda_2c_1+4c_2$.
\end{proposition}

\begin{proof}
We are going to prove that there exist three points $x_1, x_2$ and $x$ such that $x_1, x_2 \in B_x(R)$ and the distance $d(x_1, x_2)$ is at least $R$. Consider a ball of radius $2R$ centered in $x_1$. By hypothesis, the set $B_{x_1}(2R)\setminus B_{x_1}(R)$ is non-empty, hence there exists a point $x_2$ such that $2R > d(x_1, x_2)\ge R$. The space $X$ is geodesic, hence now we can take the midpoint of $x_1x_2$ as $x$.

Denote $y_i = f(x_i)$ for $i=1, 2$.

For any point $y$ of a geodesic $(y_1, y_2) \subset Y$ there exists a point $z\in B_x(R)$ such that $d(f(z), y) \le c_1$. This follows from the fact that the image of $(x_1, x_2)$ is $c_1$-connected by the definition of a quasi-isometric embedding and every $c_1$-connected path between $y_1$ and $y_2$ includes the geodesic $(y_1, y_2)$ in its $c_1$-neighbourhood.


Now consider a chain of points $\{\tilde x_i\}$ connecting $x_1, x_2$ and such that $d(\tilde x_i, \tilde x_{i+1}) < c_1/\lambda_1$. Hence, in the image $d(f(\tilde x_i), f(\tilde x_{i+1})) < 2c_1$ and so there exists $i$ such that $d(f(\tilde x_i), y) \le 2c_1$. Notice that $Y\setminus B_y(2c_1)$ has several $(4c_1-2)$-connected components and the distance between these components is at least $4c_1$.

Suppose that a point $z$ is rather far from both $x_1$ and $x_2$: $d(z, x_i) > 4\lambda_2c_1 + c_2, i=1,2$. Suppose also that $R > 2(4\lambda_2c_1 + c_2)$ (if not there is nothing to prove). In the set $B_x(R) \setminus B_z(4\lambda_2c_1 + c_2)$ we also find a $c_1/\lambda_1$-chain. Hence, there exists a point $z' \notin B_z(4\lambda_2c_1+c_2)$ of this path such that $d(f(z'), y) \le 2c_1$. Hence, $d(f(z), f(z')) \le 4c_1$ and by property of quasi-isometry $d(z, z') \le 4\lambda_2 c_1 + c_2$, so $z' \in B_z(4\lambda_2c_1+c_2)$. This leads to a contradiction with the hypothesis of the proposition. Hence, for any $y \in (y_1, y_2)$ there exists $z' \in B_{x_1}(4\lambda_2c_1+c_2) \cup B_{x_2}(4\lambda_2c_1+c_2)$ such that $d(f(z'), y) \le 2c_1$.

Consider two points $y', y''$ on the geodesic $(y_1, y_2)$ which are close enough to each other (more precisely $d(y', y'') \le c_2/\lambda_2$) and such that respective points $z'$ and $z''$ (which minimise distances to $y'$ and $y''$, that is $d(y', f(z')) \le 2c_1$ and $d(y'', f(z'')) \le 2c_1$) lie in different balls $z' \in B_{x_1}(4\lambda_2c_1+c_2)$ and $z'' \in B_{x_2}(4\lambda_2c_1+c_2)$. So, on the one hand $d(z', z'') \ge R - 8\lambda_2c_1 - 2c_2$ and on the other hand, by triangle inequality $d(f(z'), f(z'')) \le c_2/\lambda_2 + 4c_1$. Hence $R - 8\lambda_2 c_1 - 2c_2 \le \lambda_2(c_2/\lambda_2+4c_1)+c_2 = 4\lambda_2c_1+2c_2$. So we get $R \le 12\lambda_2c_1+4c_2$.
\end{proof}


Proposition \ref{connected} implies that any quasi-isometric embedding of an $R$-ball in hyperbolic plane to a tree has distorsion at least $\sqrt{R}$. We wonder whether this conclusion is sharp.

Here is a partial answer. Let $X$ be a geodesic metric space. Here we will construct an example of a $(\sqrt R, \sqrt R, \sqrt R, \sqrt R)$-quasi-isometry of a $R$-ball in $X$ to a $\sqrt R$-ball in a tree, up to taking a $\sqrt{R}$-dense subset. In this statement the essential point is that we will consider trees of variable degree which will depend on $R$.

\begin{proposition}\label{BigTree}
Let $X$ be a geodesic metric space. For any $R > 0$ there exists a $\sqrt R$-dense subset $S(R)\subset B_{X}(R)$, a tree $T(R)$ and a $(\sqrt R, \sqrt R,$ $\sqrt R, \sqrt R)$-quasi-isometric embedding $f_R:S(R)\to T(R)$.
\end{proposition}

\begin{proof}
Consider a ball $B_{X}(R,z_0)$ centered at $z_0$. We will define a discrete set of points $S(R)$ generation by generation in the following way. The $0$-generation is the origin $z_0$. For each $k$ we pick a maximal $\sqrt{R}$-separated subset in the sphere of radius $k\sqrt R$. The resulting set $S(R)$ is $\sqrt{R}$-separated. It is also $3\sqrt{R}$-dense. Indeed, any point in $B((k+1)\sqrt{R})$ is $\sqrt{R}$-close to some point of the sphere of radius $k\sqrt R$, in which the $k$-th generation is $2\sqrt{R}$-dense, by maximality. In particular, every point of the $(k+1)$th-generation is at distance $\leq 3\sqrt R$ from at least one point of the $k$th-generation. This provides us with a tree $T(R)$ with vertex set $S(R)$: we connect each point of the $(k+1)$th-generation to a closest point of $k$th-generation (if the choice is not unique we choose the ancestor arbitrarily). Finally we set the lengths of all edges of the constructed tree $T(R)$ equal to $1$. The diameter of $T(R)$ is $\sim\sqrt R$.

Now we will sketch the proof that the induced map $f$ is a $(\sqrt R, \sqrt R, \sqrt R, \sqrt R)$-quasi-isometry. The right-hand quasi-isometric inequality $d(f(x),f(y))\leq O(\sqrt{R})d(x,y)$ is automatically verified because the diameter of $T(R)$ is $O(\sqrt{R})$. Conversely, given points $x, y\in S(R)$, $z_0$, $f(x)$ and $f(y)$ form a tripod we median point $u$. The distance $d(f(x),f(y))$ is achieved by an arc from $f(x)$ to $u$ followed by an arc from $u$ to $f(y)$ in the tree. The descending arcs from $f(x)$ to $u$ (resp. from $f(y)$ to $u$) consist of jumps in $S(R)$ from generation to generation, each of distance at most $3\sqrt{R}$. Therefore $d(x,y)\leq 3\sqrt{R}d(f(x),f(y))$. 
\end{proof}

In the same manner as in the previous proposition we can construct a $(\sqrt R, \sqrt R, \sqrt R, \sqrt R)$-quasi-isometry between a ball $B_T(R)$ of radius $R$ in a regular tree $T$ of degree $d\ge 2$ and a $\sqrt R$-dense subset in a ball $B_{\HH^2}(k\ln d, z_0)$ in $\HH^2$.

\begin{proposition}\label{denseSet}
For any $R > 0$, there exist a $\sqrt R$-dense subset $S_R$ of a ball $B_{\HH^2}(R)$ in the hyperbolic plane $\HH^2$ and a $(\sqrt R,\sqrt R,\sqrt R,\sqrt R)$-quasi-isometry $f_R:B_T(R)\to B_{\HH^2}(R)$.
\end{proposition}

\begin{proof}
First we will construct the set $S_R$ and the quasi-isometry $f_R$ and then we will prove that it is indeed a $(\sqrt R,\sqrt R,\sqrt R,\sqrt R)$-quasi-isometry. Consider $k-th$ generation $G_k$ of vertices in $B_T$ (that is, points at distance $k$ from the base point), there are $(d+1)d^{k-1}$ points in it. Consider a circle centered in $z_0$ of radius $R_k$ (its exact value will be calculated soon) and take a subset $S_k$ of this circle consisting of $(d+1)d^k$ points, such that distance between them is at least $\sqrt R_k$. So we have the following relation (up to some multiplicative constants) which appears from the consideration of volumes
$$Vol(\text{ball of radius }\sqrt{R_k})(d+1)d^k = Vol(\text{circle of radius $R_k$}).$$
For big $R_k$ we have approximately
$$e^{\sqrt{R_k}}(d+1)d^k = e^{R_k}.$$
We set $R_0 = 0$. Then it follows that $R_k \approx k\ln d$. We send points from $G_k$ to $S_k$ naturally. Now we need to add edges between points of successive sets $S_k$. We connect points of $S_k$ to the nearest points from $S_{k-1}$. If there are two possibilities, we choose one arbitrary.

Let us show that this is a $(\sqrt R, \sqrt R, \sqrt R, \sqrt R)$-quasi-isometry. First of all, for any two points $t_1, t_2 \in S$, the distance between their images is at least $\sqrt R$. We have always $d(t_1,t_2) \le R \le \sqrt R d(f_R(t_1),f_R(t_2))+\sqrt R$ and this inequality is checked automatically. Now, let $u_0 = t_1, u_1, \ldots, u_{n-1}, u_n = t_2$ be a geodesic path between $t_1$ and $t_2$. We notice that $d(u_i,u_{i+1})=1 \ge d(f(u_i),f(u_{i+1}))/\sqrt R$ for $i=0,1,\ldots,n-1$. Then
\begin{eqnarray*}
d(t_1,t_2) = \sum_{i=0}^{n-1}d(t_i,t_{i+1}) &\ge& \sum_{i=0}^{n-1}d(f(u_i),f(u_{i+1}))/\sqrt R \ge
\\
d(f(t_1),f(t_2))/\sqrt R &\ge& \left(d(f(t_1),f(t_2))-\sqrt R\right)/\sqrt R,
\end{eqnarray*}
what finishes the proof.

\end{proof}

Though we do not know if we can extend this quasi-isometry to the whole ball $B_{\HH^2}(R)$. The first idea is to do a projection of $B_{\HH^2}(R)$ on a discrete subset, but this projection is a $(1, 1, \sqrt R, \sqrt R)$-quasi-isometry itself, hence the resulting map is a $(R, R, R, R)$-quasi-isometry.

\section{Poincar\'e inequalities and quasi-isometries}

\subsection{The critical exponent $p_{\neq 0}$ for $L^p$-cohomology}

$L^p$-cohomology groups provides invariants for quasi-isometries. 
The continuous first $L^p$-cohomology group of a hyperbolic metric space $X$ is
$$L^pH^1_{cont}(X) := \left\{[f] \in L^pH^1(X)| f \text{ extends continuously to } X \cup \partial X\right\},$$
where $X\cup \partial X$ is Gromov's compactification of $X$.
Following the works of Pierre Pansu, and Marc Bourdon and Bruce Kleiner \cite{BourdonKleiner}, we define the following quasi-isometrical numerical invariant of $X$
$$p_{\neq 0}(X)=\inf\left\{p\ge1| L^pH^1_{cont}(X) \neq 0\right\}.$$

If $p_{\neq 0}$ achieves different values for two spaces $X$ and $Y$, then $X$ and $Y$ are not quasi-isometric. We expect that the difference $|p_{\neq 0}(X)-p_{\neq 0}(Y)|$ also bounds from below the quasi-isometrical distortion growth. We are able to prove this only for a family of examples, and under certain restrictions on maps.

Let $Z_\mu$ and $Z_{\mu'}$ be two variants of the space $\T^n\times(-\infty, \infty)$ with metrics $dt^2+\sum e^{2\mu_i t}dx_i^2$ and $dt^2+\sum e^{2\mu_i' t}dx_i^2$ respectively. The main result of this part is a sharp lower bound for the quasi-isometrical distortion growth between $Z_\mu$ and $Z_{\mu'}$, of the form
$$const\left(p_{\neq 0}(Z_{\mu'}) - p_{\neq 0}(Z_{\mu})\right)R.$$

\subsection{Definition of Poincar\'e constants}

Constants in Poincar\'e inequalities are the quantitative incarnation of $L^p$-cohomology. On Riemannian manifolds, Poincar\'e inequality is defined as follows.
\begin{definition}
Let $X$ be a Riemannian manifold. We say that $X$ satisfies \textit{Poincar\'e inequality} if there exists a real number $C$ such that for any real valued function $f$ on $X$, there exists a real number $m_f$ such that
$$||f - m_f||_p \le C\,||\nabla f||_p.$$
The best constant $C$, denoted by $C_p(X)$, is called \textit{Poincar\'e constant} of $X$.
\end{definition}

We are not satisfied by this definition as we want to work with a wider class of metric spaces. The generalization involves semi-norms induced by kernels (see Definitions \ref{kernel}, \ref{seminorm}). Let $\psi$ be a kernel on $X$. The semi-norm $N_{p, \psi}(f) $ is an analog of the $\Lp$-norm of the gradient on a Riemannian manifold.

\bigskip

First we recall what are kernels on geodesic metric spaces.

\begin{definition}\label{kernel}
Let $X$ be a geodesic space, $dx$ a Radon measure on $X$. A kernel $\psi$ is a measurable non-negative function on $X\times X$ such that
\begin{itemize}
\item $\psi$ is bounded, $\psi\leq S^{\psi}$;
\item for every $x \in X$ $\int_X \psi(x, x') dx' = 1$;
\item the support of $\psi$ is concentrated near the diagonal: there exist constants $\varepsilon^{\psi} > 0$, $\tau^{\psi} > 0$ and $R^{\psi} < \infty$ such that $\psi(x, y) > \tau^{\psi}$ if $d(x,y) \le \varepsilon^{\psi}$; $\psi(x, y) = 0$ if $d(x,y)>R^{\psi}$.
\end{itemize}
$R^{\psi}$ is called the \emph{width}, $\varepsilon^{\psi}$ - the \emph{radius of positivity}, $S^{\psi}$ - the \emph{supremum} and $\tau^{\psi}$ - the \emph{margin} of $\psi$.
\end{definition}

\begin{definition}\label{cocycle}
A \emph{cocycle} on $Y$ is a measurable map $a: Y\times Y \to \R$ such that for every $y_1, y_2, y_3$ in $Y$,
$$a(y_1, y_2) = a(y_1, y_3) + a(y_2, y_3).$$
The convolution of a cocycle with a kernel is defined by
$$a\ast \phi(x,x') = \int_{Y\times Y}a(y,y')\phi(x,y)\phi(x',y')\,dy\,dy'.$$
\end{definition}

\begin{definition}\label{seminorm}
Let $\psi$ be a kernel and $a$ a cocycle on $X$. The semi-norm $N_{p, \psi}$ is defined by
$$N_{p, \psi}(a) = \left(\int_{X\times X} |a(x_1,x_2)|^p\psi(x_1,x_2)\,dx_1 \,dx_2\right)^{1/p}.$$
For $f$ a measurable function on $X$,
$$N_{p, \psi}(f) = \left(\int_{X\times X} |f(x_1)-f(x_2)|^p\psi(x_1,x_2)\,dx_1 \,dx_2\right)^{1/p}.$$
\end{definition}

\begin{definition}
The Poincar\'e inequality associated with a kernel $\psi$ is 
$$||f - m_f||_p \le C_p(X,\psi) N_{p, \psi} (f).$$
\end{definition}

\subsection{Scheme of proof of a lower bound on distorsion}

For the family of spaces $Z_\mu$, it is known that $p_{\neq 0}(Z_\mu)=\frac{\sum\mu_i}{\max\mu_i}$ (unpublished result of P. Pansu). In Theorem \ref{lowerBound} we show that
\begin{itemize}
\item if $p>p_{\neq 0}(Z_\mu)$, then the Poincar\'e constant for a ball of radius $R$ satisfies
$$C_p(B^{Z_\mu}(R)) \ge const.(Vol B(R))^{1/p};$$
\item if $p\leq p_{\neq 0}(Z_\mu)$, then
$$C_p(B^{Z_\mu}(R)) = o\left((Vol B(R))^{1/p}\right).$$
\end{itemize}
Next, we show that under transport by a $(\lambda,c)$-quasi-isometry, $C_p$ is multiplied by at most $e^{(\lambda+c)/a}$ for some positive constant $a$. Transport under quasi-isometric embeddings is more delicate, this is why our arguments work only for a family of examples. For these examples, we are able to get a lower bound. Roughly speaking, it states

\emph{Assume that $p_{\neq 0}(Z_{\mu'}) <p< p_{\neq 0}(Z_{\mu})$. If there exists a $(\lambda,c)$-quasi-isometric embedding $B^{Z_\mu}(R)\to Z_{\mu'}$, which induces an isomorphism on fundamental groups, then 
$$C_p(B^{Z_\mu}(R)) \geq const. e^{-(\lambda+c)/a}C_p(B^{Z_{\mu'}}(R)).$$}

This yields
\begin{eqnarray*}
\lambda+c &\geq& a(\log(C_p(B^{Z_{\mu'}}(R)))-\log(C_p(B^{Z_{\mu}}(R)))\\
&\sim&(p_{\neq 0}(Z_{\mu'}) - p_{\neq 0}(Z_{\mu}))R.
\end{eqnarray*}
which is the announced lower bound on quasi-isometric distortion growth.

\section{Regularisation and quasi-isometries}

In this section we will study how Poincar\'e inequalities are transformed under quasi-isometries. For this purpose we will use kernels, which will help us to regularize transported functions. 

\subsection{Kernels}

The convolution of two kernels is
$$\psi_1 \ast \psi_2 = \int_X \psi_1(x, z)\psi_2(z, y)\,dz,$$
the result is also a kernel. The convolution of a kernel and a function is
$$g\ast\psi(x)= \int_{X}g(z)\psi(x,z)\,dz.$$

\begin{lemma}\label{kernelWidth}
There exists a constant $c_{\tau}$ (which depends on the local geometry of the space $X$) such that for any $\varepsilon>0$ there exists $\tau = c_\tau e^{-\varepsilon}$ and a kernel $\psi$ on $X\times X$ such that for any two points $x_1, x_2$ with $d(x_1, x_2) < \varepsilon$, we have $\psi(x_1, x_2) > \tau$. In other words, for any given radius of positivity $\varepsilon$ there exists a kernel with a margin controlled from below by $c_\tau e^{-\varepsilon}$.
\end{lemma}

\begin{proof}
We start from kernel 
$$\psi'(x,x')=Vol(B(x,1))^{-1}1_{\{d(x,x')\leq 1\}}$$ 
with radius of positivity $\varepsilon'=1$ and margin $\tau'=v(1)^{-1}$, where, for $r>0$, $v(r)$ denotes the infimum of volumes of balls of radius $r$ in $X$. We know from the proof of Lemma 1.2 in \cite{Pansu2} that the $m$-th convolution $\psi'^{\ast m}$ has radius of positivity $\varepsilon'_m \ge m(\varepsilon'/2)=m/2$ and margin $\tau'_m \ge \tau'^m v(\frac{1}{2})^{m-1}$. We denote $v(\frac{1}{2})^{m-1}$ by $c_\tau$ which finishes the proof.
\end{proof}

The following facts are known, see \cite{Pansu2}.

\begin{lemma}\label{seminorms}
Let $X$ be a geodesic metric space such that the infimum $\inf\{Vol B(x,r)|x\in X\}$ of volume of balls of radius $r$ is positive. Semi-norms $N_{p, \psi}$ are pairwise equivalent. More precisely, let $\psi_1$ and $\psi_2$ be two kernels on $X$. Then
$$N_{\psi_2} \le \hat C N_{\psi_1},$$
where
$$\hat C = \frac{\sup\psi_1\sup\psi_2}{c_\tau} \frac{R^{\psi_2}}{\varepsilon^{\psi_1}} (2e)^{R^{\psi_2}/\varepsilon^{\psi_1}}.$$
\end{lemma}

\begin{lemma}\label{grad2cocycle}
Let the space $X$ be a Riemannian manifold and have the following properties: (1) its injectivity radius is bounded below, (2) its Ricci curvature is bounded from below. Then the volumes of balls are bounded from below (Croke inequality \cite{Croke}) and from above (Bishop inequality).

1) For any function $g$ define a cocycle $u(x, y) = g(x) - g(y)$. Then for any $p$ and any kernel $\psi'$ with bounded derivatives there exists a kernel $\psi_1$ such that the $\LL^p$-norm of $\nabla (g\ast\psi')$ (we regularise $g$) is bounded from above by a $\psi_1$-seminorm of the corresponding cocycle $u$
$$||\nabla (g\ast \psi')||_p \le N_{p, \psi_1}(u)$$
with the kernel $\psi_1$ defined as follows
$$\psi_1 = \frac{\sup\nabla\psi'\sup\psi'}{Vol(B(z', R^{\psi'}))}1_{\{d(z, z') \le R^{\psi'}\}}.$$

2) Conversely, there also exists a kernel $\psi_2$ such that
$$N_{p, \psi_2}(u) \le C||\nabla g||_p,$$
where $C$ depends only on dimension. Here the kernel $\psi_2$ can be taken as
$$\psi_2(x, y) = \max\{1, \Theta(x, y)^{-1}\}1_{\{d(x,y) \le R\}},$$
where $\Theta(x, y)$ is the density of the volume element in polar coordinates with origin at $x$
$$\Theta(x,y)^{-1}dy = drd\theta$$
and $R>0$ can be chosen arbitrarily.
\end{lemma}

In the third hypothesis we propose to use $R = 1$, then $\psi_2$ is bounded by $1$ and the width of its support is also $1$. For reader's convenience, we include the proof of the first statement of the last Lemma, following \cite{Pansu2}.
\begin{proof}
Denote by $\alpha$ the cocycle $u \ast \psi'$. Then for any $y$,
$$\nabla(u\ast\psi')(x) = \frac{\partial \alpha(x, y)}{\partial x} = \int \left(g(z') - g(z)\right)d_x\psi'(z, x)\psi'(z', y)\,dz\,dz'.$$
Choose $y = x$. Then we obtain
$$|\nabla(g \ast\psi'(x))| \le \sup\nabla\psi'\sup\psi \int_{B(x, R^\psi)\times B(x, R^\psi)}|g(z') - g(z)|\,dz\,dz'.$$
Now applying H\"older inequality we get the needed statement with the kernel
$$\psi_1 = \frac{\sup\nabla\psi'\sup\psi'}{Vol(B(z', R^{\psi'}))}1_{\{d(z, z') \le R^{\psi'}\}}.$$
\end{proof}

This lemma gives us an idea how to generalize Poincar\'e inequalities for the case of arbitrary metric spaces. Of course, such Poincar\'e inequality depends on a choice of a kernel $\psi$. Let $f$ be an $\LL_p$-function on $X$, $\psi$ a kernel on $X$. The Poincar\'e inequalities for $f$ associated to $\psi$ with constants $c_f$ and $C_p(f)$ is
$$||f - c_f||_p \le C_p(f)||N_{p,\psi}(u)||.$$
The Poincar\'e constant $C_p(X,\psi)$ is a constant such that for any $\LL_p$-function $f$ Poincar\'e inequality is checked with $C_p(f) = C_p(X,\psi)$. It follows from Lemma \ref{seminorms} that the existence of Poincar\'e constant does not depend on the choice of a kernel.

\subsection{Transporting functions by quasi-isometries}

Let $X, Y$ be two metric spaces, let $f:X\to Y$ and $f':Y\to X$ be $(K, c)$-quasi-isometries between them such that for any $x\in X$, $d(x, f'\circ f(x)) \le c$ and vice versa (that is, they are inverse in the quasi-isometrical sense). Let $g$ be a measurable function on $Y$. We want to find a way to transport and to regularize $g$ by our quasi-isometry to obtain a similar measurable function on $X$. We will take
$$h(x) = \int_Y g(z) \psi(f(x), z) \,dz$$
as a function on $X$ corresponding to $g$. This integral exists for all $x$ because $\psi$ is measurable by the second variable by definition. Still we want $h$ to be also measurable. For that, it will be sufficient if $f$ is measurable too.

\begin{proposition}
Let $f$ be a $(\lambda_1,\lambda_2, c_1,c_2)$-quasi-isometric embedding between metric spaces $X$ and $Y$. Then there exists a measurable $(\lambda_1,\lambda_2,3c_1,c_2+2c_1/\lambda_1)$-quasi-isometric embedding $g$ at distance $2c_1$ from $f$.
\end{proposition}

\begin{proof}
Take a measurable partition $P$ of $X$ with a mesh $c_1/\lambda_1$. For each set $A\in P$ we choose a base point $x_A$. We set $g$ be constant on $A$
$$g|_A = f(x_A).$$

Take any two points $x, x' \in X$. Assume $x \in A$ and $x' \in A'$ where $A, A' \in P$. Then
\begin{eqnarray*}
d(g(x),g(x')) &=& d(f(x_A),f(x_{A'})) \le \lambda_1d(x_A, x_A') + c_1
\\
&\le&\lambda_1(d(x, x')+d(x,x_A)+d(x',x_{A'}))+c_1 \le \lambda_1d(x,x')+3c_1.
\end{eqnarray*}
In the same way we prove the right-hand inequality.
\end{proof}
This proposition gives us an idea that we can always pass to measurable quasi-isometries without significant loss in constants. From now we will consider only measurable quasi-isometries.

\subsection{Transporting cocycles}

\begin{definition}
Let $a$ be a cocycle on $Y$, $f:X\to Y$ be a quasi-isometric embedding and $\phi$ be a kernel on $Y$. The transporting convolution of $a$ with $\phi$ by $f$ is the cocycle defined on $X$ by
$$a\ast_t \phi(f)(x,x') = \int_{Y\times Y}a(y,y')\phi(f(x),y)\phi(f(x'),y')\,dy\,dy'.$$
\end{definition}

\begin{lemma}\label{transpCocycles}
Let $X, Y$ be two metric space. Suppose also that $X$ has a bounded geometry (that is for any $R > 0$ the supremum of volume of balls of radius $R$ in $X$ is bounded). Let $\phi$ be a kernel on $Y$, let $a$ be a cocycle on $Y$ and let $\psi$ be a kernel on $X$. Let also $f$ be a $(\lambda_1,\lambda_2,c_1,c_2)$-quasi-isometric embedding. Then there exists a kernel $\tilde\psi$ on $Y$ such that 
$$N_{\psi}(a\ast_t\phi(f))\le CN_{\tilde\psi}(a),$$ 
where $$C \le \left(c_\tau^Y\right)^{-1} e^{R^{\psi'}} \sup\psi \left(\sup\phi \sup Vol B_X(2\lambda_2R^\phi + c_2)\right)^2.$$


\end{lemma}

\begin{proof}
By definition,
\begin{eqnarray*}
(N_{\psi}(a\ast_t\phi(f)))^p = \int_{X\times X}|a \ast_t\phi(x, x')|^p\psi(x,x')dxdx'=
\\ =\int_{X\times X}
\left|
\int_{Y\times Y}a(y,y')\phi(f(x),y)\phi(f(x'),y')dydy'
\right|^p\psi(x,x')dxdx'
\end{eqnarray*}
applying H\"older inequality
$$\le \int_{X\times X}\int_{Y\times Y}|a(y,y')^p|\phi(f(x),y)\phi(f(x'),y')dydy' \psi(x,x')dxdx'$$
denoting $\psi'(y,y') = \int_{X\times X}\phi(f(x),y)\phi(f(x'),y')\psi(x,x')dxdx'$
$$=\int_{Y\times Y}|a(y,y')|^p\psi'(y,y')dydy'.$$
We need to show that $\psi'$ is dominated by some kernel $\tilde\psi$.

First we will prove that $\psi'(y,y') = 0$ if $d(y,y') > R^{\psi'}$ for some $R^{\psi'} = R^{\phi} + \lambda R^{\psi'} + c$.
If $d(x,x') > R^\psi$ then by the definition of kernels $\psi(x,x') = 0$. Otherwise, suppose that $d(x, x') < R^{\psi'}$. If $d(y, y') > R^{\psi'}$, then by triangle inequality either $\phi(f(x), y)$ or $\phi(f(x'), y')$ vanishes:
$$d(f(x), f(x')) \le \lambda d(x,x') + c \le \lambda R^{\psi'} + c.$$
Hence, if, for example, $d(f(x), y) \le R^{\phi}$, then $d(f(x'), y') \ge R' - d(f(x), f(x')) > R^{\phi}$ which leads to $\phi(f(x'), y') = 0$.

We estimate $\psi'(y,y')$ from above in the following way. First we write
$$\psi'(y,y') \le \sup\psi \int_{X\times X}\phi(f(x),y)\phi(f(x'),y')dxdx'.$$
Now we have to integrate $\int_{X}\phi(f(x),y)dx$ and $\int_{X}\phi(f(x'),y')dx'$.

For any $y \in Y$, if $d(f(x), y) > R^\phi$ then $\phi(f(x), y) = 0$. Hence, the diameter of the set of points $X_y \in X$ such that for any $x \in X_y$ $d(f(x), y) \le R^\phi$, is at most $\lambda_22R^\phi + c_2$. Hence, $\int_X \phi(f(x), y)dx \le \left(\sup_{x\in X} Vol B_X(x,2\lambda_2R^\phi + c_2)\right) \sup_{Y\times Y}\phi$, that is $\sup_{x\in X} Vol B_X(x,2\lambda_2R^\phi + c_2)$ stands for the supremum of volumes of all balls of radius $2\lambda_2R^\psi + c_2$ in $X$. So we come to the following upper-bound for $\psi'(y,y')$
$$\psi'(y,y') \le \sup\psi \left(\sup\phi \sup Vol B_X(2\lambda_2R^\phi + c_2)\right)^2.$$

Lemma \ref{kernelWidth} helps us to construct a kernel $\tilde\psi$ such that its radius of positivity is at least $R^{\psi'}$ and at the same time we control its margin from below. $\tilde\psi(y,y') \ge \tau = c_\tau^Y e^{-R^{\psi'}}$ whenever the distance between $y, y'$ does not exceed $R^{\psi'}$. Hence, 
$$\psi'(y,y') \le \tau^{-1} \tilde\psi(y, y') \sup\psi \left(\sup\phi \sup Vol B_X(2\lambda_2R^\phi + c_2)\right)^2.$$
So, we obtain
$$C \le \left(c_\tau^Y\right)^{-1} e^{R^{\psi'}} \sup\psi \left(\sup\phi \sup Vol B_X(2\lambda_2R^\phi + c_2)\right)^2.$$

\end{proof}

\section{Poincare inequality for exponential metric}
\label{poin}

We will give an upper bound for the Poincar\'e constant in a ball of radius $R$ in a space with the metric $dt^2 + \sum_i e^{2\mu_i t}dx_i^2$.


\begin{theorem}\label{expPoincare}
Let $\tilde X = \R_+\times\R^n$ with the metric $dt^2 + \sum_i e^{2\mu_i t}dx_i^2$. Let $X = \tilde X/\Gamma$ where $\Gamma$ is a lattice of translations in the factor $\R^n$. Then the Poincar\'e constant for a ball $B(R)$ in $X$ is
$$C_p(\mu) \le \frac{p}{\mu}+(A(\mu))^{1/p}C_p(\T^n)e^{\mu_n R},$$
where $\mu = \sum\mu_i$, $A(\mu)$ is a constant depending only on $\mu$, $C_p(\T^n)$ is a Poincar\'e constant for a torus $\T^n$.
\end{theorem}

First, we fix the direction $\theta = (x_1, \ldots, x_n)$.

\subsection{Poincar\'e inequality in a fixed direction}
\begin{lemma}
Let $\tilde X = \R_+\times\R^n$ with the metric $dt^2 + \sum_i e^{2\mu_i t}dx_i^2$. Let $X = \tilde X/\Gamma$ where $\Gamma$ is a lattice of translations in the factor $\R^n$. Let $R\in\R^+\cup\{\infty\}$. Then for any fixed direction $\theta = (x_1, \ldots, x_n)$
$$\left(\int_a^{R}|f(t) - c_\theta|^pe^{\mu t}dt\right)^{1/p} \le \frac{p}{\mu}\left(\int_a^{R}|f'(t)|^pe^{\mu t}dt\right)^{1/p},$$
where $c_\theta = f(R,\theta)$ or $c_\theta = \lim_{R\to\infty}f(R,\theta)$.
\end{lemma}

\begin{proof}
Let $f$ be a function such that its partial derivative $\partial f/\partial t$ is in $\LL^p(e^{\mu t}dt,[0,+\infty))$ where $p > 1$. By H\"older inequality we get
$$\int_0^{+\infty}\left|\frac{\partial f}{\partial t}\right|dt \le \left(\int_0^{+\infty}\left|\frac{\partial f}{\partial t}\right|^pe^{\mu t}dt\right)^{1/p}\left(\int_0^{+\infty}e^{-(\mu t/p)(p/(p-1))}\right)^{1-1/p} < +\infty.$$
Hence, for every fixed direction $\theta$ there exists a limit $\lim_{t \to \infty} f(t, \theta)$.

First, if $R = \infty$, we prove that $|f(t) - c_\theta|^pe^{\mu t} \to 0$ as $t \to \infty$. We apply the Newton-Leibniz theorem and then H\"older inequality to $|f(t) - c_{\theta}|$. We have
\begin{eqnarray}
|f(t) - c_\theta| = \left|\int_t^{\infty}\frac{\partial f}{\partial s}ds\right| \le \int_t^{\infty}\left|\frac{\partial f}{\partial s}\right|ds \le
\\\nonumber
\le \left(\int_t^{\infty}\left|\frac{\partial f}{\partial s}\right|^pe^{\mu u}du\right)^{1/p}\left(\int_t^{\infty}e^{-\mu s/(p-1)}ds\right)^{1 - 1/p}.
\end{eqnarray}
We calculate the last integral
$$\int_t^{\infty}e^{-\mu s/(p-1)}ds = -\frac{p-1}{\mu}e^{-\frac{\mu s}{p-1}}|_t^{\infty} = \frac{p-1}{\mu}e^{-\frac{\mu t}{p-1}}.$$
With the notation $D_0 = \left(\frac{p-1}{\mu}\right)^{p-1}$,
$$|f(t) - c_\theta|^p \le D_0e^{-\mu t}\int_t^{+\infty}\left|\frac{\partial f}{\partial s}\right|^p e^{\mu s}ds.$$
Hence
$$|f(t) - c_\theta|^pe^{\mu t} \le D_0\int_t^{+\infty}\left|\frac{\partial f}{\partial s}\right|^p e^{\mu s}ds \to 0$$
as $t \to +\infty$.

Now we integrate by parts
\begin{equation}\label{byParts}
\int_a^R|f(t) - c_\theta|^pe^{\mu t}dt = \left[|f(t) - c_\theta|^p \frac{e^{\mu t}}{\mu}\right]_a^R - \int_a^Rf'(t)p|f(t) - c_\theta|^{p-1}\frac{e^{\mu t}}{\mu}dt.
\end{equation}
As $c_\theta = f(R)$
$$\int_a^R|f(t) - c_\theta|^pe^{\mu t}dt = -|f(a) - c_\theta|^p \frac{e^{\mu a}}{\mu} - p\int_a^Rf'(t)|f(t) - c_\theta|^{p-1}\frac{e^{\mu t}}{\mu}dt.$$
We notice that the integral at the left is positive. On the right hand side, the first term is negative (for this reason we will drop it soon). Hence, the second term should be positive.
By H\"older inequality,
\begin{equation}\label{holder}
\int_a^{R}(-f'(t))|f(t) - c_\theta|^{p-1}\frac{e^{\mu t}}{\mu}dt \le \left(\int_a^{R}|f'(t)|^p\frac{e^{\mu t}}{\mu}dt\right)^{1/p}\left(\int_a^{R}|f(t) - c_\theta|^p\frac{e^{\mu t}}{\mu}dt\right)^{(p-1)/p}.
\end{equation}
We introduce the following notations
\begin{eqnarray*}
X = \int_a^{R}|f(t) - c_\theta|^pe^{\mu t}dt,
\quad
Y = \int_a^{R}|f'(t)|^pe^{\mu t}dt.
\end{eqnarray*}
Using these notations we return to Eq. (\ref{byParts}). First we drop the term $-|f(a) - c_\theta|^p e^{\mu a}/\mu$ and then we apply Eq. (\ref{holder}) 
$$X \le \frac{p}{\mu}Y^{1/p}X^{(p-1)/p}.$$
So, we get immediately that
$$X^{1/p} \le \frac{p}{\mu}Y^{1/p}$$
which proves Poincar\'e inequality in a fixed direction.
\end{proof}

\subsection{Poincar\'e inequality for exponential metric}

Here we will finish the proof of Theorem \ref{expPoincare}. We introduce the following notations $\tilde f_r(t,\theta) = f(r,\theta)$ (the function is considered as a function of two variables), $f_r(\theta) = f(r, \theta)$ (the function is considered as a function of one variable).

We have already proved that for any $\theta \in \T^n$,
$$\int_0^R|f(t,\theta)-f(R,\theta)|^pe^{\mu t}dt \le \left(\frac{p}{\mu}\right)^p\int_0^R\left|\frac{\partial f}{\partial t}\right|^pe^{\mu t}dt.$$
We integrate over $\theta$ and we introduce the volume element for $\tilde X$, $dVol = drd\theta e^{\sum\mu_i r}$. We get
$$\int_{B(R)}|f - f_R|^pdVol \le \left(\frac{p}{\mu}\right)^p\int_{B(R)}|\nabla f|^pdVol.$$

Denote the Euclidean gradient by $\nabla_e$. By the form of the metric we see that $e^{2\mu_i t}|dx_i^2| = 1$. Hence, $||\nabla_e f_r|| \le e^{\mu_n t}|\nabla f|$. Now we notice that
$$\int_{R-1}^R||\nabla_e f_r||^p_{\LL^p(\T^n)}e^{\mu t}dt \ge e^{\sum\mu_i(R-1)}\int_{R-1}^R||\nabla_e f_r||^p_{\LL^p(\T^n)}dt.$$
So we write
\begin{equation}\label{int1}
e^{\sum\mu_i(R-1)}\int_{R-1}^R||\nabla_e f_r||^p_{\LL^p(\T^n)}dt \le e^{p\mu_nR}\int_{B(R)\backslash B(R-1)}|\nabla f|^pdVol.
\end{equation}
Fixing $r \in [R-1, R]$, let us write Poincar\'e inequality on the torus for the function $f_r(\theta)$. There exists a number $c_r$ such that
$$\int_{\T^n}|f_r(\theta) - c_r|^pd\theta \le (C_p(\T^n))^p\int_{\T^n}|\nabla_e f_r(\theta)|^pd\theta,$$
where $C_p(\T^n)$ is a Poincar\'e constant for $\T^n$. Next we consider the function $f_r(\theta)$ as a function on the ball $B(R)$ which does not depend on $t$. We integrate this inequality over $t$,
\begin{eqnarray*}
\int_{B(R)}|f_r(\theta) - c_r|^pdVol &\le& (C_p(\T^n))^p\int_{0}^{R}\int_{\T^n}|\nabla_e f_r(\theta)|^pd\theta e^{\sum\mu_i t}dt
\\
&\le& \frac{e^{\sum \mu_i R}}{\sum\mu_i}(C_p(\T^n))^p\int_{\T^n}|\nabla_e f_r(\theta)|^pd\theta.
\end{eqnarray*}
We integrate over $r$ from $R-1$ to $R$ and exploit inequality (\ref{int1}). It gives
$$
\int_{R-1}^R\left(\int_{B(R)}|f_r(\theta) - c_r|^pdVol\right)dr \le A(\mu)(C_p(\T^n))^pe^{p\mu_nR}\int_{B(R)\backslash B(R-1)}|\nabla f|^pdVol,
$$
where $A(\mu)$ is a constant which depends only on $\mu_i, i=1,\ldots, n$. Now we apply H\"older inequality again,
\begin{eqnarray*}\int_{R-1}^{R}||f_r-c_r||_{\LL^p(B(R))}dr &\le& \left(\int_{R-1}^R\int_{B(R)} |f_r-c_r|^pdVol\,dr\right)^{1/p}
\\
&\le& \left(A(\mu)(C_p(\T^n))^pe^{p\mu_nR}\int_{B(R)\backslash B(R-1)}|\nabla f|^pdVol\right)^{1/p}
\\
&\le& \left(A(\mu)\right)^{1/p}C_p(\T^n)e^{\mu_nR}||\nabla f||_{\LL^p(B(R))}
\end{eqnarray*}

Set $c = \int_{R-1}^R c_r dr$. In the following chain of inequalities we will first apply triangle inequality and then we will use the fact that the norm of the integral is less than or equal to the integral of the norm (briefly $||\int fdr|| = \int||f||dr$).
\begin{eqnarray*}
||f - c||_{\LL^p(B(R))} &=& \left\|\int_{R-1}^R(f-c_r)dr\right\|_{\LL^p(B(R))}
\\
&\le& \left\|\int_{R-1}^R(f-f_r)dr\right\|_{\LL^p(B(R))} + \left\|\int_{R-1}^R(f_r-c_r)dr\right\|_{\LL^p(B(R))}
\\
&\le& \int_{R-1}^R\left(||f-f_r||_{\LL^p(B(R))} + ||f_r-c_r||_{\LL^p(B(R))}\right)dr 
\\
&\le& \frac{p}{\mu}||\nabla f||_{\LL^p(B(R))} + \left(A(\mu)\right)^{1/p}C_p(\T^n)e^{\mu_nR}||\nabla f||_{\LL^p(B(R))}.
\end{eqnarray*}

\section{Lower bound on Poincar\'e constant}

Let $Z_\mu$ denote $\T^n\times\R$ equipped with metrics $dt^2 + \sum e^{2\mu_i t}dx_i^2$, where we suppose $\mu_1 \le \mu_2 \le \ldots \le \mu_n$. Let $O$, $O'=(0,\ldots,0)$ be base points of $Z$ and $Z'$ respectively. We notice that the "width" of $\T^n\times(-\infty, 0]$ is finite so it is at finite distance from a ray $(-\infty, 0]$, so from now on, we shall focus our attention on the part of $B_Z (O,R)$ where $t\ge 0$. Indeed, we want to consider quasi-isometric embeddings of balls $\T^n\times[-R,R]$. The volume of $T^n\times(-\infty,0]$ is finite, whereas the volume of $\T^n\times[0,R]$ is exponential in $R$. Hence, only a negligible part of $\T^n\times[-R,R]$ can be sent to the negative part $\T^n\times(-\infty,0]$ (compare to subsection \ref{volumes}).

Consider a ball $B_Z(O,R)$ in $Z=Z_\mu$ and its quasi-isometric embedding in $Z'=Z_{\mu'}$. In this section we will give a lower bound for the sum of quasi-isometric constants $\lambda+c$ in function of $R$, using our results on transported Poincar\'e inequalities.
We have to notice that our method does not apply to a general quasi-isometric embedding. We will consider only quasi-isometric embeddings which are homotopy equivalences.

Why do we want to consider these spaces $Z_\mu$? Following U.HamenstŠ\"adt \cite{Hamen} and X. Xie \cite{Xie},\cite{SXie}, there is a family of hyperbolic spaces whose quasi-isometric classification is known, that are spaces with transitive Lie groups of isometries. In this family (classified by E.Heintze \cite{Heintze}), the easiest spaces are $X_\mu$. We also know their $\Lp$ cohomologies (Pansu, \cite{Pansu2}). Still they are rather difficult because their $\Lp$ cohomology vanishes for a delicate global reason, which is hard to make quantitative, on balls. Fortunately, their quotients $Z_{\mu}$ by $\Z^n$ are simpler. We can also say that the spaces $Z_{\mu}$ are hyperbolic spaces with ideal boundaries being products of circles supplied with power of the standard metric.

\subsection{Statement of theorems}

\begin{theorem}\label{lowerBoundPoincare}
Let $Z, Z'$ be two locally homogeneous hyperbolic metric spaces with metrics $dt^2 + \sum e^{2\mu_i t}dx_i^2$ and $dt^2 + \sum e^{2\mu_i' t}dx_i^2$ respectively, $0 < \mu_1 \le \mu_2 \le \ldots \le \mu_n$ and $0< \mu_1' \le \mu_2' \le \ldots \le \mu_n'$. Assume also that $\sum\mu_i/\mu_n > \sum\mu_i'/\mu_n'$. Suppose that there exist constants $a$ and $b$ such that for any $i$, $b \le \mu_i, \mu_i' \le a$. Then there exists a constant $G_0(a,b)$ such that the following holds.
Let $\Theta: B_Z(R) \to Z'$ be a continuous $(\lambda_1,\lambda_2,c_1,c_2)$-quasi-isometric embedding, inducing an isomorphism on fundamental groups. Suppose that $\Theta$ sends base point to base point, $\Theta(O) = O'$ and that $R \ge 8(\lambda_1+c_1)+(\lambda_2+c_2)+1$. If $p > \sum\mu_i'/\mu_n'$, up to  replacing $Z$ with a connected $2$-sheeted covering, the Poincar\'e constant $C_p(\mu)$ for a ball of radius $R$ in the space $Z$ is bounded from below by
$$C_p(\mu) \ge \left(G_0(a,b)\right)^{1/p}\left(\lambda_1+c_1\right)^{-3/p-2/p^2}e^{-(9/p+3/p^2)(\lambda_1+c_1)}e^{(\sum\mu_i/p)R}\left(p-\sum\mu_i'/\mu_n'\right)^{1/p}.$$
\end{theorem}

This theorem is not symmetric, it can be applied only in one direction: it does not give any lower bound to the quasi-isometric embeddings of $Z_\mu$ to $Z_{\mu'}$ and of $Z_{\mu'}$ to $Z_{\mu}$ at the same time.

As we have already mentioned, we are able to treat the quantitative problem only for quasi-isometric embeddings which are homotopy equivalences. So we modify Definition \ref{DG} in the following way.

\begin{definition}\label{hDG}
Let $X, Y$ be metric spaces, $x_0, y_0$ their base points respectively. The homotopy quasi-isometric distortion growth is the function
\begin{eqnarray*}
D_{hG}(X,x_0,Y,y_0)(R) = \inf\{d|\exists f:B_X(x_0,R) \to Y \text{ a $(\lambda_f,c_f)$-quasi-isometric embedding}
\\\text{such that }
f(x_0)=y_0
\text{ and } f \text{ is a homotopy equivalence}, d=\lambda_f+c_f \}.
\end{eqnarray*}
\end{definition}

\begin{theorem}\label{lowerBound}
Let $Z, Z'$ be two locally homogeneous hyperbolic metric spaces with metrics $dt^2 + \sum e^{2\mu_i t}dx_i^2$ and $dt^2 + \sum e^{2\mu_i' t}dx_i^2$ respectively, $0 < \mu_1 \le \mu_2 \le \ldots \le \mu_n$ and $0< \mu_1' \le \mu_2' \le \ldots \le \mu_n'$. Assume also that $\sum\mu_i/\mu_n > \sum\mu_i'/\mu_n'$. Suppose that there exist constants $a$ and $b$ such that for any $i$ $b \le \mu_i, \mu_i' \le a$. Then there exist constants $G_1(a, b)$ and $G_2(a, b)$ such that the following holds.
The homotopy distortion growth (see Definition \ref{hDG}) for quasi-isometrical embedding of  $B_Z(R)$ into $Z'$ is bounded from below by
$$D_{hG} (R)\ge \min\left\{G_1\left(\frac{\sum\mu_i}{\mu_n} - \frac{\sum\mu_i'}{\mu_n'}\right)R - G_2,\frac18R\right\}.$$
\end{theorem}

Theorem \ref{lowerBoundPoincare} plays an important role in the proof of Theorem \ref{lowerBound}. Before proving these two theorems, we will discuss the double cover of the family of spaces under consideration and we will give some preliminary lemmas.

\subsection{Lifting to a double covering space}
Let us introduce a double covering of $Z'$. Let $\tilde{Z}' = \R^{n-1}/\Z^{n-1} \times \R/2\Z \times [0,+\infty)$ with the metric defined by the same formula as for $Z'$: $dt^2 + \sum e^{2\mu_it}dx_i^2$. Consider the map $\tilde{Z}' \to Z'$ defined by
$$(x_1, x_2, \ldots, x_n, t) \mapsto (x_1, x_2, \ldots, x_n \mod 1, t).$$
So we identify $(x_1, x_2, \ldots, x_n, t)$ and $(x_1, x_2, \ldots, x_n+1, t)$ in $\tilde{Z}'$. Consider a complex function $u(x_1, x_2, \ldots, x_n, t) = e^{\pi i x_n}$ on $\tilde{Z}'$.

Composition of $u$ with the deck transformation $\iota': \tilde{Z}' \to \tilde{Z}'$
$$\iota': (x_1, x_2, \ldots, x_n, t) \mapsto (x_1, x_2, \ldots, x_n+1, t)$$
gives $u\circ\iota' = -u$.

By assumption, $\Theta: Z\to Z'$ is a continuous map inducing an isomorphism on fundamental groups, and we have $\tilde Z'$ which is a covering space of $Z'$. We need to show that there exists a non-trivial covering space $\tilde Z \to Z$ such that the following diagram commutes.
$$\begin{array}{rcl} \tilde Z & \xrightarrow{\tilde\Theta} & \tilde Z'\\
\pi_Z\downarrow &  & \downarrow\pi_{Z'}\\
Z & \xrightarrow{\Theta} & Z'\end{array}$$
Define
$$\tilde Z = \left\{(z, \tilde z')|z\in Z, \tilde z'\in\pi_{Z'}^{-1}(\Theta(z))\right\},$$
that is $\tilde Z\subset Z\times\tilde Z'.$ Let $[\gamma']$ be a loop in $Z'$ which does not lift to a loop in $\tilde Z'$. By hypothesis, there exists a loop $\gamma$ in $Z$ such that $\Theta(\gamma)$ is homotopic to $\gamma'$. Then $\gamma$ does not lift to a loop in $\tilde Z$. There exists an isometry $\iota$ of order $2$ on $\tilde Z$ such that $\tilde\Theta\circ\iota = \iota'\circ\tilde\Theta$.

\subsection{Lifting of $\Theta$}

Here we will prove that in the constructed double coverings $\Theta$ lifts to a map satisfying the right-hand inequality in the definition of quasi-isometry with constants $\lambda_1$ and $2c_1$. We need two preliminary lemmas concerning distances in two-fold coverings.

\begin{lemma}\label{curv}
Let $Z=Z_\mu$ be a locally homogeneous space. There is an effective constant $c_0(\mu)$ with the following effect. Let $z$ be a point in $Z$ in the region where $t\geq c_0$. Let $c=t(z)$. Every loop of length less than $c$ based at $z$ is null-homotopic.
\end{lemma}

\begin{proof}
Let $\pi_s:Z\to T^n\times\{s\}\subset Z$ denotes projection onto the first factor. This is a homotopy equivalence. Note that $\pi_s$ is length decreasing on $\{(t,x)\in Z\,;\,t\geq s\}$. Moreover, on $T^n\times\{t\}$, $\pi_s$ decreases length by $e^{\mu_1(s-t)}$ at least.
Let $\gamma$ be a non null-homotopic geodesic loop at $z$. Assume that its length is $\leq 2c$. Then $\gamma\subset \{(t,x)\in Z\,;\,t\geq \frac{c}{2}\}$, therefore
\begin{eqnarray*}
\mathrm{length}(\pi_{\frac{c}{2}}(\gamma))\leq c,
\end{eqnarray*}
thus
\begin{eqnarray*}
\mathrm{length}(\pi_{0}(\gamma))\leq c \,e^{-\mu_1\frac{c}{2}}.
\end{eqnarray*}
Since $\pi_{0}(\gamma)$ is not null-homotopic, its length is at least 1, and this shows that
\begin{eqnarray*}
c\geq e^{\mu_1\frac{c}{2}}.
\end{eqnarray*}
This can happen only for $c\leq c_0(\mu_1)$.
\end{proof}

\begin{lemma}
Let $z_1, z_2$ be two points in $Z$ such that $d(O', \Theta(z_1)) > c_1$ or $d(O', \Theta(z_2)) > c_1$ and $d(z_1, z_2) \le c_1/\lambda_1$. Then $d(\tilde\Theta(\tilde z_1),\tilde\Theta(\tilde z_2)) = d(\Theta(z_1),\Theta(z_2))$.
\end{lemma}

\begin{proof}
Let $\tilde z_1\in\tilde Z$ be such that $d(\tilde O,\tilde z_1) > c_1$. Set
\begin{eqnarray*}
W &=& \{\tilde z_2\in \tilde Z | , d(\tilde z_1,\tilde z_2) \le c_1\},
\\
U &=&\{\tilde z_2\in W | d(\tilde\Theta(\tilde z_1),\tilde\Theta(\tilde z_2)) = d(\Theta(z_1),\Theta(z_2))\} \subset W,
\\
V &=&\{\tilde z_2\in W | d(\tilde\Theta(\tilde z_1),\iota'\circ\tilde\Theta(\tilde z_2)) = d(\Theta(z_1),\Theta(z_2))\} \subset W.
\end{eqnarray*}
By construction, $W = U \cup V$. Let us show that the intersection of $U$ and $V$ is empty
$$U\cap V =\{\tilde z_2\in W | d(\tilde\Theta(\tilde z_1),\iota'\circ\tilde\Theta(\tilde z_2)) = d(\tilde\Theta(\tilde z_1),\tilde\Theta(\tilde z_2))\}.$$
If $\tilde z_2\in U\cap V$, then the geodesic segments connecting $\tilde\Theta(\tilde z_1)$ with $\tilde\Theta(\tilde z_2)$ and $\tilde\Theta(\tilde z_1)$ with $\iota'\circ\tilde\Theta(\tilde z_2)$ induce a loop $\gamma$ in $Z'$ of length $2d(\Theta(z_1),\Theta(z_2)) \le 2\left(\lambda_1(c_1/\lambda_1)+c_1\right) = 4c_1$ which is not homotopic to $0$. According to Lemma \ref{curv}, this is incompatible with the assumption that $d(O',\Theta(z_1)) > c_1$. Hence, $U\cap V$ is empty. Since $U$ is non-empty (it contains at least $\tilde z_1$) and closed in $W$, $V$ is closed in $W$ and $W$ is connected, we conclude that $U = W$, which finishes the proof.
\end{proof}

\begin{lemma}
A $(\lambda_1,\lambda_2,c_1,c_2)$-quasi-isometric embedding $\Theta: Z\to Z'$ lifts to a ``quasi-Lipschitz" map $\tilde\Theta:\tilde Z\to\tilde Z'$, that is, for any two points $\tilde z_1,\tilde z_2 \in \tilde Z$,
$$d(\tilde\Theta(\tilde z_1),\tilde\Theta(\tilde z_2)) \le \lambda_1d(\tilde z_1,\tilde z_2) + 2c_1.$$
\end{lemma}

\begin{proof}
Let $\tilde\gamma \subset \tilde Z$ be a geodesic between $\tilde z_1$ and $\tilde z_2$. Let $t_1$ be the first point such that $d(\tilde\Theta\gamma(t), \tilde O') \le c_1$ and $t_2$ be the last point with such a property (if such points $t_1, t_2$ do not exist, then we can apply the following arguments directly to $d(\tilde\Theta(\tilde z_1),\tilde\Theta(\tilde z_2))$ instead of cutting the curve in three parts and considering $d(\tilde\Theta(\tilde z_1),\tilde\Theta\tilde\gamma(t_1))+d(\tilde\Theta(\tilde z_1),\tilde\Theta\tilde\gamma(t_2))$). Then
$$d(\tilde\Theta(\tilde z_1),\tilde\Theta(\tilde z_2)) \le
d(\tilde\Theta\tilde\gamma(t_1),\tilde\Theta\tilde\gamma(t_2))+
d(\tilde\Theta(\tilde z_1),\tilde\Theta\tilde\gamma(t_1))+
d(\tilde\Theta(\tilde z_1),\tilde\Theta\tilde\gamma(t_2)).$$
By definition of $t_1$ and $t_2$, $d(\tilde\Theta\tilde\gamma(t_1),\tilde\Theta\tilde\gamma(t_2)) \le 2c_1$. Now divide parts of $\gamma$ between $\tilde\Theta(\tilde z_1)$ and $\tilde\Theta\tilde\gamma(t_1)$ and between $\tilde\Theta(\tilde z_1)$ and $\tilde\Theta\tilde\gamma(t_2)$ by segments of length $c_1/\lambda_1$. We apply the previous lemma to them, so
$$d(\tilde\Theta(\tilde z_1),\tilde\Theta\tilde\gamma(t_1))+
d(\tilde\Theta(\tilde z_1),\tilde\Theta\tilde\gamma(t_2)) \le N\left(\lambda_1\frac{c_1}{\lambda_1}+c_1\right),$$
where $N \le d(\tilde z_1, \tilde z_2)/(c_1/\lambda_1)$ is the number of segments in the subdivision. So,
$$d(\tilde\Theta(\tilde z_1),\tilde\Theta(\tilde z_2)) \le 2c_1+2\lambda_1d(\tilde z_1,\tilde z_2).$$
\end{proof}

\subsection{Proof of Theorem \ref{lowerBoundPoincare} - Part 1}
Let $\psi'$ be a kernel on $\tilde Z$ which is invariant by isometry, that is, for any isometry $\iota$, $$\psi'(\iota(\tilde z_1), \iota(\tilde z_2)) = \psi'(\tilde z_1, \tilde z_2).$$
As an example of such a kernel we can consider a kernel depending only on the distance between points.
Let $\zeta$ be a kernel on $\tilde Z'$ which is also invariant by isometries. Define a complex function $v$ on $\tilde Z$ as follows
$$v(\cdot) = \left(\int _Y u(\tilde z')\zeta(\tilde \Theta(\tilde z),\tilde z')d\tilde z'\right)\ast\psi'(\cdot, \tilde z).$$
We will write shortly for the integral
$$u\ast_t\zeta(\tilde \Theta)(\tilde z) = \int _Y u(\tilde z')\zeta(\tilde \Theta(\tilde z),\tilde z')d\tilde z'.$$
Then $v\circ\iota = -v$. Indeed,
$$v\circ\iota = \left(u\ast_t\zeta(\tilde \Theta)\right)\ast\psi'\circ\iota = \left(u\ast_t\zeta(\tilde \Theta)\circ\iota\right)\ast\psi'.$$
On the other hand, using both relations $\tilde\Theta\circ\iota = \iota'\circ\tilde\Theta$ and $(\iota')^2 = id$, we have
\begin{eqnarray}\nonumber
u\ast_t\zeta(\tilde \Theta)\circ\iota(\tilde z) = \int u(\tilde z')\zeta(\tilde \Theta(\iota \tilde z), \tilde z')d\tilde z' = \int u(\tilde z')\zeta(\iota'\tilde\Theta(\tilde z),(\iota')^2\tilde z')d\tilde z' =
\\\nonumber
= \int u(\tilde z')\zeta(\tilde\Theta(\tilde z),\iota' \tilde z')d\tilde z' = \int u(\iota'\tilde z')\zeta(\tilde\Theta(\tilde z),\tilde z')d\tilde z' = - u\ast_t\zeta(\tilde \Theta),
\end{eqnarray}
hence $v$ is skewsymmetric with respect to $\iota$. We get immediately that $\int v = 0$. Now we apply successively Lemma \ref{grad2cocycle} and Lemma \ref{transpCocycles}.

\emph{Step 1.} By Lemma \ref{grad2cocycle} there exists a kernel $\psi_1$ on $\tilde Z$ which is controlled by $a$ and $b$ and such that
$$\left(\int|\nabla(u\ast_t\zeta(\tilde\Theta)\ast\psi')|^p\right)^{1/p} \le N_{\psi_1}\left(u\ast_t\zeta(\tilde\Theta)\right),$$
where for $\psi_1$ we have the width of support is $R^{\psi_1} = R^{\psi'}$ and
$$\sup\psi_1 \le \frac{\sup\nabla\psi'\sup\psi'}{\inf_{z} Vol B(\tilde z, R^{\psi})}.$$

\emph{Step 2.} By Lemma \ref{transpCocycles} there exists a kernel $\zeta_1$ on $\tilde Z'$ such that
$$N_{\psi_1}\left(u\ast_t\zeta(\tilde\Theta)\right) \le \tilde CN_{\zeta_1}(u),$$
where the width of support of $\zeta_1$ is $2R^{\zeta}+\lambda_1R^{\psi'}+c_1$, the supremum of $\zeta_1$ is
$$\sup\zeta_1 = \frac{\sup\psi_1}{c_{\tau}^Y} e^{2R^\zeta+\lambda_1R^{\psi'}+c_1}(2\lambda_1R^\zeta+c_1)^2$$
and
$$\tilde C = \frac{1}{c_\tau^Y}(\sup\psi_1)^{3/p}e^{\left((2+\lambda_1)R^{\psi'}+c_1\right)/p}\left((2+\lambda_1)R^{\psi'}+c_1\right)^{2/p}.$$

\emph{Step 3.} 
Applying Lemma \ref{grad2cocycle} we get that there exists a kernel $\zeta_2$ on $\tilde Z'$ such that
$$N_{\zeta_2}(u) \le C(n)||\nabla u||_p,$$
we remind that the constant $C(n)$ depends only on the dimension of $\tilde Z'$ if the Ricci curvature is bounded from below, that is $\sup\mu_i$ is bounded.

\emph{Step 4.}
Here we merely need to pass from $N_{\zeta_1}$ to $N_{\zeta_2}$. We apply Lemma \ref{grad2cocycle} once more
$$N_{\zeta_1} \le \hat C N_{\zeta_2},$$
where
$$\hat C = \frac{\sup\zeta_1\sup\zeta_2}{c_\tau^Y} \frac{R^{\zeta_2}}{\varepsilon^{\zeta_2}} (2e)^{(2R^{\zeta}+\lambda_1R^{\psi'}+c_1)/\varepsilon^{\zeta_2}}.$$

Choose $\psi'$ and $\zeta$ such that $R^{\psi'}=1$ and $R^{\zeta}=1$. Then $\sup\psi'$ and $\sup\zeta$ are controlled by $a$ and $b$. We note also that $\varepsilon^{\zeta_2}=1$. So combining all inequalities we get
$$\int_{B(R)} |\nabla v|^p \le C_1(a, b)\left(\lambda_1+c_1\right)^{3+2/p}
e^{(9+3/p)(\lambda_1+c_1)} \int_{\T^n\times[0,+\infty]}|\nabla u|^p,$$
where $C_1(a,b)$ is a constant depending only on $a$, $b$ and dimension $n$. Let $Q = \lambda_1 + c_1$ and
$$C(Q) = \left(\lambda_1+c_1\right)^{3+2/p}e^{(9+3/p)(\lambda_1+c_1)}.$$

\subsection{Proof of Theorem \ref{lowerBoundPoincare}}

We will give a lower bound for the $\LL^p$-norm  of the function $v=(u\ast \phi) \ast \psi'$. Our aim is to prove that the absolute value of $v$ is nearly constant. For simplicity of notations we suppose first that the volume growth of $Z_{\mu}$ and $Z_{\mu'}$ is the same, that is $\sum\mu_i=\sum\mu'_i$. We will write $|\mu|$ and $|\mu'|$ for these sums respectively. We are going to show that there exists a subset $A$ of the ball $B(z_0, R)$ such that on the one hand the volume of $A$ is rather big, that is $Vol(A) \ge  Vol(B(z_0, R))/2$ and on the other hand its image lies rather far from the base point $\Theta(A) \cap B(z'_0, R - (\lambda_1+c_1+\lambda_2+c_2)) = \emptyset$.

Denote by $r = \lambda_2+c_2$. We will construct a finite subset $J$ in $B(z_0, R)\subset Z_{\mu}$ and a partition of $J$ into $e^{|\mu|r}$ subsets $\{J_k\}_{k=1,\ldots,n}$, each of cardinality $|J_k| = e^{|\mu|(R-r)}$ with the following property
\begin{itemize}
\item {\bf (P)} For any $k\in\{1,\ldots,n\}$ if $z_1$ and $z_2$ are points of $J_k$ then the open balls of radius $r$ centered at these points are disjoint.
\end{itemize}

So, let $z_1, z_2\in J_k$ be two different points. It follows from {\bf (P)} that
$$2r \le d(z_1,z_2) \le \lambda_2d(\Theta(z_1),\Theta(z_2))+c_2,$$
hence $d(\Theta(z_1),\Theta(z_2)) \ge 2$, so the balls $B(\Theta(z_1),1)$ and $B(\Theta(z_2),1)$ are disjoint. Fix some $d>0$ and denote by $J_k' \subset J_k$ the set of points whoes images are not farther than $R-d$ from $z'_0$ that is if $z \in J'_k$ then $d(z'_0,\Theta(z)) \le R-d$. We obtain
$$|J'_k| Vol(B(\Theta(z), 1)) \le Vol(B(z'_0,R-d+1)),$$
and we conclude that $|J'_k| \le e^{|\mu|(R-d)}.$ Denote the union of $J'_k$ by $J'$ then $|J'| \le e^{|\mu|(R-d+r)}$. Hence, whenever $d \ge r+1$,
$$\frac{|J'|}{|J|} \le e^{|\mu|(r-d)} \le \frac{1}{2}.$$
So, we choose $d = r+1$. Now let $A$ be the union of all $1$-balls centered at points of $J\setminus J'$, $A = \cup_{z\in J\setminus J'}B(z,1)$. The volume $Vol A \ge 1/2 Vol(B(z_0,R))$. By definition of $A$, for any point $z\in A$ there exists a point $z' \in J\setminus J'$ at most $1$-far away from $z$, $d(z,z') \le 1$. Applying triangle inequality we get $d(z'_0,\Theta(z)) \ge d(z'_0,\Theta(z')) - (\lambda_1 +c_1) \ge R - (\lambda_1+c_1+\lambda_2+c_2)$.

Here we describe the set  $J\subset\{R\}\times\R^n/\Z^n$ (we fix the first coordinate $t = R$). This is the set of points $z =(R, x_1,\ldots,x_n)$ such that for any $i=1,\ldots,n$, $x_i$ is an integer multiple of $e^{-\mu_iR}$ modulo $1$. $J_0$ is the subset of points such that for any $i$, $x_i$ is a whole multiple of $e^{\mu_i(r-R)}$. Let $K$ be the set of vectors $k=(0, k_1, \ldots, k_n)$ such that for any $i$ the number $e^{\mu_iR}k_i$ is an integer between $0$ and $e^{\mu_i(r-R)}-1$. For $k \in K$, we define $J_k = J_0 + k$. Then for any two different points $z_1, z_2$ of $J_k$,
$$d(z_1,z_2) =\max\log\left(|x^1_i-x^2_i|^{1/\mu_i}\right)\ge r.$$
We constructed the needed set. Now we notice that the lifting $\tilde A \subset \tilde Z$ of $A$ has the same properties relatively to $\tilde\Theta$: the image $\tilde\Theta(\tilde A)$ lies at distance at least $R - (\lambda_1+c_1+\lambda_2+c_2)$ from the base point and the volume of $\tilde A$ is at least a half of the volume of the ball $B(\tilde z_0, R)$.

Now let us compute $|v(\tilde z)|$ for $\tilde z \in \tilde A$ (in fact here we will give an upper bound for $|v|$ which is true for all $\tilde z \in B(z_0,R)$ and a lower bound for $\tilde z \in \tilde A$). We remind that by construction, $\tilde z$ is sent far from the base point, $d(\tilde z'_0, \tilde\Theta(\tilde z)) \ge R - (\lambda_1+c_1+\lambda_2+c_2)$.
\begin{eqnarray}\nonumber
|(u\ast \phi) \ast \psi'(\tilde z)| &=& \left|\int_X\int_Y u(\tilde z')\zeta(\tilde \Theta(\tilde z_1),\tilde z')\psi'(\tilde z,\tilde z_1)d\tilde z'd\tilde z_1\right|
\\\nonumber
&\ge& \left|\int_X\int_Y (u(\tilde z')-u(\tilde \Theta(\tilde z)) + u(\tilde \Theta(\tilde z)))\zeta(\tilde \Theta(\tilde z_1),\tilde z')\psi'(\tilde z,\tilde z_1)d\tilde z'd\tilde z_1\right| 
\\\nonumber
&\ge& \left|\int_X\int_Y (u(\tilde \Theta(\tilde z)))\zeta(\tilde \Theta(\tilde z_1),\tilde z')\psi'(\tilde z,\tilde z_1)d\tilde z'd\tilde z_1\right|
\\\nonumber
&&- \left|\int_X\int_Y (u(\tilde z')-u(\tilde \Theta(\tilde z)))\zeta(\tilde \Theta(\tilde z_1),\tilde z')\psi'(\tilde z,\tilde z_1)d\tilde z'd\tilde z_1\right| 
\\\label{integral1}
&\ge& 1 - \int_X\int_Y |u(\tilde z')-u(\tilde \Theta(\tilde z))|\zeta(\tilde \Theta(\tilde z_1),\tilde z')\psi'(\tilde z,\tilde z_1)d\tilde z'd\tilde z_1.
\end{eqnarray}
For the last inequality we shall use the following facts: $|u| = 1$ and the integral of a kernel over the second argument is equal to $1$.
\begin{eqnarray*}
&&\left|\int_X\int_Y u(\tilde \Theta(\tilde z))\zeta(\tilde \Theta(\tilde z_1),\tilde z')\psi'(\tilde z,\tilde z_1)d\tilde z'd\tilde z_1\right| 
\\
&=& \left|\int_X u(\tilde \Theta(\tilde z))\psi'(\tilde z,\tilde z_1)\left(\int_Y\zeta(\tilde \Theta(\tilde z_1),\tilde z')d\tilde z'\right)d\tilde z_1\right|
\\
&=& \left|\int_X u(\tilde \Theta(\tilde z))\psi'(\tilde z,\tilde z_1)d\tilde z_1\right| = \left|u(\tilde \Theta(\tilde z))\right| = 1.
\end{eqnarray*}

We need to estimate the double integral in Eq. (\ref{integral1}). $\psi'(\tilde z,\tilde z_1)$ is non-zero if $d(\tilde z,\tilde z_1) \le R^{\psi'} = 1$ and $\zeta(\tilde \Theta(\tilde z_1),\tilde z')$ is non-zero if $d(\tilde z',\tilde \Theta(z_1)) \le R^{\zeta} = 1$. So the diameter of the set $\hat S$ of points $\tilde z'$ such that the integrand is non-zero, is at most $2\lambda_1+c_1+2\le4(\lambda_1+c_1)$ because $\lambda_1 \ge 1$. Hence $\hat S$ is contained in a ball $B_{\hat S}$ of radius $4(\lambda_1+c_1)$. Assume $\hat z' = \tilde \Theta(\tilde z) \in \hat S$. Then by the mean value theorem, for any point $\tilde z'\in \hat S$,
\begin{eqnarray}\nonumber
|u(\tilde z') - u(\hat z')| \le |\tilde z'-\hat z'|\sup_{\tilde z'\in B_{\hat S}}|\nabla u(\tilde z')|\le 8(\lambda_1+c_1)\sup_{\tilde z'\in B_{\hat S}}\left|\frac{\partial u}{\partial \tilde x_n}\right|e^{-\mu_n't} \le 8\pi(\lambda_1+c_1)e^{-\mu'_nt}
\\\nonumber
\le8\pi(\lambda_1+c_1)\sup_{\hat z'\in B_{\hat S}}e^{-\mu'_nd(O',\hat z')}\le 8\pi(\lambda_1+c_1)e^{-\mu'_n\left(R - (\lambda_1+c_1+\lambda_2+c_2) - 2(\lambda_1+c_1)\right)} \le \frac12
\end{eqnarray}
for $R \ge 8(\lambda_1+c_1)+(\lambda_2+c_2) = R_0$.
Hence we have proved that
\begin{eqnarray*}
\frac12 \le |(u\ast \phi) \ast \psi'(\tilde z)|&if& \tilde z \in \tilde A
\\
|(u\ast \phi) \ast \psi'(\tilde z)| \le 1&if& \tilde z \in B(\tilde z_0, R).
\end{eqnarray*}
And we conclude from this relation that for $R \ge R_0+1$,
$$\int_{B(R)}|v|^p \ge \frac{1}{2^p}Vol(B(R)) - Vol(B(R_0)) \ge e^{(\sum\mu_i)R}/2^{p+1}.$$
Let us compute the integral $\int |\nabla u|^p$.
$$\int |\nabla u|^p = \int \left|\frac{\partial u}{\partial x_n}\right|^pe^{-\mu_n'pt}e^{\left(\sum\mu_i'\right)t}dtdx_n = \pi \int_0^{+\infty}e^{\left(\sum\mu_i' - p\mu_n'\right)t}dt = \frac{\mu_n'\pi}{-\sum\mu_i'/\mu_n'+p}.$$

Hence the Poincar\'e constant $C_p(\mu)$ for $Z$ satisfies
\begin{eqnarray*}
(C_p(\mu))^p &\ge& \frac{||v||^p}{||\nabla v||^p} \ge \frac{||v||^p}{C_1(a,b)C(Q)||\nabla u||^p}\\
&\ge& \left(\mu_n'\pi2^{p+1} C_1(a,b)C(Q)\right)^{-1}e^{(\sum\mu_i)R}(p-\sum\mu_i'/\mu_n').
\end{eqnarray*}
This proves the claim in Theorem \ref{lowerBoundPoincare}.

\subsection{Proof of Theorem \ref{lowerBound}}

Let $\Theta:B_Z(R)\to Z'$ be a $(\lambda_1,\lambda_2,c_1,c_2)$-quasi-isometric embedding. By hypothesis, $\Theta$ is isomorphic on fundamental groups. Lemma \ref{curv} implies that $\Theta$ moves the origin a bounded distance away. Indeed, a non null-homotopic loop of length $1$ based at $O$ is mapped to a non null-homotopic loop of length $\le Q=\lambda_1+c_1$ based at $\Theta(O)$. This implies that $t(\Theta(O))\leq 4Q$ and $d(O',\Theta(O))\leq 4Q+1$.

The space $\tilde Z$ is of the form $\tilde T\times\R$ where $\tilde T \to T$ is a connected $2$-sheeted covering space of torus, that is $\tilde T$ is also a torus. Hence we can apply Theorem \ref{expPoincare}. We have $C_p(\mu) \le C_2(a, b)e^{\mu_n R}$. If $R \le 8(\lambda_1+c_1)+(\lambda_2+c_2)$ there is nothing to prove. Otherwise we arrive at
$$\left(\mu_n'\pi2^{p+1} C_1(a,b)C(Q)\right)^{-1/p}e^{(\sum\mu_i/p)R}\left(p-\sum\mu_i'/\mu_n'\right)^{1/p} \le C_2(a, b)e^{\mu_n R}.$$
Hence with $C_3(a, b) = (\mu_n'\pi2^{p+1}C_1(a,b))^{1/p}C_2(a, b)$,
$$C_3(a, b)C(Q) \ge e^{(\sum\mu_i/p - \mu_n)R}\left(p - \frac{\sum\mu_i'}{\mu_n'}\right)^{1/p}.$$
We have calculated that $C(Q) = Q^{3+2/p}e^{(9+3/p)Q}$. Combining these results and taking the logarithm (note that in the following calculations every constant depending on $\mu$ and $\mu'$ can be estimated using $a$ and $b$), we get
$$\left(3+\frac2p\right)\log Q + \left(9+\frac3p\right)Q \ge G'(a,b)+\left(\frac{\sum\mu_i}{p}-\mu_n\right)R+\frac1p\log\left(p-\frac{\sum\mu'_i}{\mu'_n}\right)$$
with some constant $G'$ depending only on $a,b$. $p\ge 1$ hence  the left-hand size can be estimated as $5\log Q + 12Q < 24 Q$. Setting $p = \sum \mu_i'/\mu_n' + 1/R$, we get
$$24 Q \ge G'(a,b)+\frac{\mu_n\left(\frac{\sum\mu_i}{\mu_n}-\frac{\sum\mu'_i}{\mu'_n}-\frac{1}{R}\right)R}{\frac{\sum\mu'_i}{\mu'_n}+\frac{1}{R}} + \frac1p\log\frac1R.$$
For $R \ge G''(a, b)$ with some well-chosen constant $G''$,
$$24 Q \ge G'(a,b)+\frac{\mu_n\mu'_n}{4\sum\mu'_i}\left(\frac{\sum\mu_i}{\mu_n}-\frac{\sum\mu'_i}{\mu'_n}\right)R - \frac{\mu'_n}{2\sum\mu'_i}\log R,$$
and finally we can rewrite our inequality under the desired form
$$Q \ge G_1(a, b)\left(\frac{\sum\mu_i}{\mu_n} - \frac{\sum\mu_n'}{\mu_n'}\right)R - G_2(a,b)$$
with $G_1(a, b)$ and $G_2(a, b)$ being constants depending only on $a$ and $b$.

This finishes the proof of Theorem \ref{lowerBound}.

\section{Quasi-isometric distortion for regular trees}

In this section, we prove that embedding hyperbolic balls into trees requires linear distorsion growth.

First we need coarse notions of volume and of separation (minimal volume of subsets dividing a metric space $X$ into two pieces). 

\begin{definition}
Let $a > 0$. We will call the $a$-volume of a metric space $X$ the following quantity
$$Vol_a(X) = \sup\left\{v\big|\text{for any family }B_j\text{ of balls of radius }a\text{ covering }X: \#\{B_j\}\ge v\right\}.$$
\end{definition}

\begin{definition}
Let $a > 0$. We call $a$-separation of $X$ the number
\begin{eqnarray*}
sep_a(X) &=& \sup\big\{N \big| \text{for any partition } X = U_1 \sqcup U_2\text{ such that }Vol_a(U_i)\ge Vol_a(X)/3,
\\ &&i=1,2,\, \text{for any family } B_j \text{ of pairwise disjoint balls of radius } a,
\\ &&\#{\text{balls intersecting both } U_1 \text{ and } U_2}\ge N\big\}
\end{eqnarray*}
\end{definition}

\begin{theorem}\label{linearTree}
Let $X$ be a bounded metric space, and $T$ be a tree of degree at most $d$. $S = sep_a(X)$ and $V = Vol_a(X)$. Suppose that for any subset $Y$ of $X$ of $a$-volume at least one third of $V$, the diameter of $Y$ is at least $diam(X)/D$ for some constant $D$ depending only on $X$. If $f:X\to T$ is a $(\lambda, c)$-quasi-isometric embedding then
\begin{itemize}
  \item either $diam (X) \le cD$,
  \item or
$$\lambda 2a + c \ge \log_d\frac{S}{Vol_a(B(c))}.$$
\end{itemize}
\end{theorem}

\begin{proof}
Let $\{B_j\}$ be a maximal set of pairwise disjoint balls of radius $a$. We consider $T$ as a finite discrete metric space. If there exists a vertex $t$ of $T$ such that at least one third of centers of $B_j$ are sent to $t$ then $diam (X) \le cD$ because of the hypothesis on the space $X$. Otherwise, for any vertex $t$,
$$Vol_a(f^{-1}(t)) < Vol_a(X)/3.$$

We are going to find a vertex $t$ which divides the tree into two components $T = T_1 \cup T_2, T_1\cap T_2 = \{t\}$ such that $Vol_a f^{-1}(T_i)$, $i=1,2$, is at least one third of $Vol_a(X)$. To show this, it suffices to start from some boundary vertex (we will call $T_1$ the component which contains the initial vertex) of the tree and to pass from one vertex to another.  At every step we choose a vertex which increases $Vol_af^{-1}(T_1)$. We finish when the accumulated volume is sufficient, that is $Vol_af^{-1}(T_1) \ge Vol_a(X)/3$.

Denote by $U_i = f^{-1}(T_i), i=1,2$. The number $N_s$ of balls $B_j$ which intersect both $U_1$ and $U_2$ is at least $N_s \ge sep_a(X) = S$. Let $I$ be a set which contains a point of the intersection $U_1\cap B_j$ for all such balls, denote the image of $I$ by $I' = f(I)$.
$$|I'| \ge \frac{S}{Vol_a(B(c))}.$$
Because $I' \subset T_1$, there exists $v_1 \in I'$ such that
$$d(v_1, T_2) \ge \log_d\frac{S}{Vol_a(B(c))}.$$
Thus $v_1 = f(u_1)$ and $u_1 \in B_j$ which intersects $U_2$, there exists $u_2 \in U_2\cap B_j$ and $d(u_1, u_2) \le 2a$, hence $d(v_1, f(u_2)) \le \lambda2a+c$. Hence,
$$\lambda2a+c \ge \log_d\frac{S}{Vol_a(B(c))}.$$
\end{proof}

Consider $\HH^n$ with $n\ge 3$. For a ball of radius $R$ in $\HH^n$ we have $S \sim e^{(n-2)R}$ (we will prove this soon, in Lemma \ref{sephn}), $V \sim e^{(n-1)R}$ and $D = 1$. Then the application of Theorem \ref{linearTree} to $B(R) \subset \HH^n$ with $n \ge 3$ proves the linear quasi-isometric distortion between $\HH^n$ and a regular tree.

\begin{corollary}
The quasi-isometric distortion growth for hyperbolic space $\HH^n$, $n\ge 3$, and a regular tree is linear in $R$.
\end{corollary}

\begin{lemma}\label{sephn}
Let $B(R):=B_{\HH^n}(R) = A_1\sqcup A_2$ be a partition of an $R$-ball of hyperbolic $n$-space. Suppose that both pieces have large volume: $Vol A_i \ge 1/3 Vol B(R)$, $i=1,2$. Then for $R$ large enough the volume of the common boundary of $A_1$ and $A_2$, $S_{12} = \partial A_1 \cap \partial A_2$ is at least 
$$Vol S_{12} \ge const(n) e^{(n-2)R},$$
where the multiplicative constant depends only on dimension $n$.
\end{lemma}

\begin{proof}
Consider the indicator function $\hat\iota$ of $A_1$
$$
\hat\iota(x) = \left\{ \begin{array}{cc}
	1 & x \in A_1 \\
	0 & x \in A_2
       \end{array} \right.
$$
We would like to write $1$-Poincar\'e inequality for $\hat\iota$ but first we have to make it smooth.

Fix two real numbers $r_1 < r_2$. Take an $r_1$-separated and $r_2$-dense set $S$. For any point $z \in S$ we define a function $h_z$ as follows,
$$
\hat h_z(x) = \left\{ \begin{array}{cc}
	1 - \frac{d(x,z)}{2R} & x\in B(z, 2r_2) \\
	0 & x\in\overline{B(z, 2r_2)}
       \end{array} \right..
$$
Now set $h_z(x) = \hat h_z(x)/\sum \hat h_{z'}(x)$ which gives a partition of unity. Evidently gradients of $h_z$ are uniformly bounded in function of $r_1, r_2$ and a number $L(r_1, r_2, n)$ of disjoint balls of radius $r_1$ in a ball of radius $2r_2$: $\nabla h_z \le N(r_1, r_2, n)$. Now we set
$$
\iota(x) = \sum_{z\in S}\iota(z)h_z(x).
$$
Then
$$\nabla\iota(x) = \sum_{z\in S}\iota(z)\nabla h_z(x) \le L(r_1, r_2, n)N(r_1, r_2, n)$$
because $\iota$ takes only two values $0$ and $1$ and for any $x$ there is not more than $L(r_1, r_2, n)$ functions $h_z$ which do not vanish at $x$. Now we notice that outside of $2r_2$-tubular neighbourhood $T_{12}$ of $S_{12} = \partial A_1\cap \partial A_2$, $\iota$ coincides with $\hat\iota$. We notice here that $Vol T_{12} \sim Vol S_{12}$ (up to some multiplicative constants depending on $n$ and $r_2$). If $Vol (A_1\setminus T_{12})$ or $Vol (A_2\setminus T_{12})$ is too small, then there is nothing to prove as $Vol T_{12} \sim Vol B(R)$. Otherwise we write Poincar\'e inequality for $\iota$.

The mean value $Vol(A_1\setminus T_{12})\le \oint\iota dVol \le Vol(A_1\setminus T_{12}) + \sup\nabla\iota\mid_{T_{12}} Vol T_{12}$, that is
$$c_{\iota} = \oint\iota dVol \sim Vol A_1.$$
We are ready to write $1$-Poincar\'e inequality for $\iota$ for $R$ large enough,
$$\int|\iota - c_{\iota}|dVol \le const(n)e^{R}\int|\nabla \iota|dVol,$$
where $const(n)$ is some constant depending only on dimension $n$. We compute the left-hand integral, 
$$\int|\iota - c_{\iota}|dVol \approx \int_{A_1}|1 - c_{\iota}|dVol + \int_{A_2}c_{\iota}dVol = (1 - c_{\iota})Vol A_1 + c_{\iota} Vol A_2 \ge \frac{2}{3}Vol B(R).$$
For the right-hand integral, we obtain
$$\int|\nabla \iota|dVol = \int_{T_{12}}|\nabla \iota|dVol \le \sup\nabla\iota Vol T_{12} = const(n) Vol S_{12}.$$

Combining all these inequalities we conclude that
$$Vol B(R) \le const(n) e^{R} Vol S_{12}.$$
As for $R$ large enough $Vol B(R) \sim e^{(n-1)R}$ we finish the proof with the needed result
$$Vol S_{12} \ge const(n) e^{(n-2)R}.$$
\end{proof}

\begin{question}
What is the quasi-isometric distortion between a $d$-regular tree and hyperbolic plane $\HH^2$.
\end{question}

\section{Approximation of distances and radial quasi-isometries}
\label{adequi}

\subsection{Orthogonal triangles in hyperbolic spaces}
At the beginning of this section we give to lemmas on the geometry of orthogonal triangles in hyperbolic spaces. The second Lemma will be used to establish an approximation of distances in hyperbolic spaces which allow to control a quasi-isometric action.
\begin{lemma}\label{orthPr1}
Let $\sigma$ be a geodesic segment, $a$ be a point not on $\sigma$, and $c$
be a projection of $a$ on $\sigma$. Let $b\in\sigma$ be arbitrary, and
let $d$ denote a projection of $b$ on $ac$. Then $|c-d|\le2\delta$.
\end{lemma}

\begin{proof}
By hypothesis, $bd$ minimizes the distance of $b$ to any point of $ac$, and
because the triangle $bcd$ is $\delta$-thin, there exists a point $e\in bd$
such that $d(e,ac)=|e-d|\le\delta$ and $d(e,bc)\le\delta$. Because $ac$
is a perpendicular to $\sigma$, $|a-c|\le|a-d|+|d-e|+d(e,bc)\le|a-d|+
2\delta$. Hence $|c-d|\le2\delta$.
\end{proof}

\begin{lemma}\label{orthTriangle}
As in the preceding lemma, let $\sigma$ be a geodesic segment, $a$ be a
point not on $\sigma$, $c$ be a projection of $a$ on $\sigma$, and $b$ be
some point on $\sigma$. Let $d$ denote a point on $ac$ such that
$|d-c|=\delta$ and $e$ denote a point on $bc$ such that $|e-c|=3\delta$.
Then
\begin{itemize}
\item $d(d,ab)\le\delta$, $d(e,ab)\le\delta$, $d(c,ab)\le2\delta$, and
\item the length of $ab$ differs from the sum of the lengths of the two
other sides by at most $8\delta$,
$$
|a-c|+|b-c|-2\delta\le|a-b|\le|a-c|+|b-c|+8\delta.
$$
\end{itemize}
\end{lemma}

\begin{proof}
The triangle $abc$ is $\delta$-thin. Therefore, obviously, $d(d,ab)\le
\delta$ (the distance from a point of $ac$ to $ab$ is a continuous
function). We take a point $x\in bc$ such that $d(x,ca)\le\delta$.
Using Lemma~\ref{orthPr1}, we obtain $|b-x|+d(x,ca)\ge|b-c|-2\delta$, and
hence $|c-x|\le d(x,ca)+2\delta\le3\delta$.

We now let $d_1$ and $e_1$ denote respective projections of $d$ and $e$
on $ab$. Then by the triangle inequality, we have
\begin{itemize}
\item $|a-d|-\delta\le|a-d_1|\le|a-d|+\delta$,
\item $|b-e|-\delta\le|b-e_1|\le|b-e|+\delta$, and
\item $0\le|d_1-e_1|\le|d_1-d|+|d-c|+|c-e|+|e-e_1|\le6\delta$.
\end{itemize}
Combining all these inequalities, we obtain the second point in the lemma.
\end{proof}

\subsection{Approximation of distances in hyperbolic metric spaces}
Let $X, Y$ be two geodesic hyperbolic metric spaces with base points $x_0 \in X$, $y_0 \in Y$. Let $\theta: \partial X  \to \partial Y$ be a homeomorphism between ideal boundaries.

\begin{hypothesis}\label{hypo1}
Assume that there exists a constant $D$ such that for any $x \in X$ there exists a geodesic ray $\gamma$ from the base point $\gamma(0) = x_0$ and passing near $x$: $d(x, \gamma) < D$.
\end{hypothesis}

We are going to construct approximately (up to $D$) a map $\Theta: X\to Y$ extending the boundary homeomorphism $\theta$. Take some point $x$ and a geodesic ray $\gamma$ from $x_0$ passing near $x$: $d(\gamma, x) < D$. Then $\gamma(\infty)$ is a point on ideal boundary $\partial X$. The corresponding point $\theta(\gamma(\infty)) \in \partial Y$ defines a geodesic ray $\gamma'$ such that $\gamma'(0) = y_0$ and $\gamma'(\infty) = \theta(\gamma(\infty))$. Set $\Theta(x) = \gamma'(d(x_0, x))$. So, by construction, $\Theta$ preserves the distance to the base point. Still, it depends on the choices of $\gamma$ and $\gamma'$.

\begin{definition}
Define the following quantity
$$K(R) = \sup\left\{\left|\log\frac{d_{y_0}(\theta(\xi_1), \theta(\xi_2))}{d_{x_0}(\xi_1, \xi_2)}\right|| d_{x_0}(\xi_1, \xi_2) \ge e^{-R} \vee d_{y_0}(\theta(\xi_1), \theta(\xi_2)) \ge e^{-R}\right\}.$$
\end{definition}

We are going to prove that the restriction of $\Theta$ on the ball $B(R) \subset X$ of radius $R$ is a $\left(1+2\frac{K(R)}{D+\delta}, D+\delta+2K(R)\right)$-quasi-isometry. We begin with a Lemma which gives an approximation (up to an additive constant) of the distance between two points in a hyperbolic metric space. In its proof, all equalities hold with a bounded additive error depending linearly on $\delta$.

\begin{lemma}\label{distApprox}
Let $P_1, P_2$ be two points in a hyperbolic metric space $Z$. Let $P_0$ be a base point (possibly at infinity). Let distances (horo-distances if $P_0$ is at infinity) from $P_1$ and $P_2$ to $P_0$ be $d(P_1, P_0) = t_1$ and $d(P_2, P_0) = t_2$. Assume that there exist points $P_1^\infty$ and $P_2^\infty$ such that $P_1$ (resp. $P_2$) belongs to the geodesic ray defined by $P_0$ and $P_1^\infty$ (resp. $P_2^\infty$). Denote by\footnote{We define $visdist(P_1^\infty, P_2^\infty)$ of two points $P_1^\infty, P_2^\infty$ at the ideal boundary as the exponential of minus Gromov's product of these points $e^{-(P_1^\infty|P_2^\infty)}$. Indeed, it is not a distance as it does not satisfy triangle inequality. But we will never have more than two points at infinity at the same time in our setting, so we will not use this property.}
$$t_\infty = -\log visdist_{P_0}(P_1^\infty, P_2^\infty)$$
the logarithm of visual distance seen from $P_0$. Then up to adding a multiple of $\delta$,
$$d(P_1, P_2) = t_1 + t_2 - 2\min\{t_1, t_2, t_\infty\}.$$
\end{lemma}

\begin{proof}
Let $P_0'$ be a projection of $P_0$ on the geodesic $P_1^\infty P_2^\infty$. By Lemma \ref{orthTriangle}, $P_0'$ lies at distance at most $2\delta$ from both $P_0P_1^\infty$ and $P_0P_2^\infty$. Hence, up to an additive constant bounded by $4\delta$ the distance between $P_0$ and $P_0'$ is equal to Gromov's product of $P_1^\infty$ and $P_2^\infty$. It follows that $t_\infty = d(P_0, P_0') = -\log visdist(P_1, P_2)$.

The triangle $P_0P_1^\infty P_2^\infty$ is $\delta$-thin. Notice that if $P_1$ (or $P_2$) lies near the side $P_1^\infty P_2^\infty$ then $t_1 \ge t_\infty$. Otherwise, $t_1 \le t_\infty$ (both inequalities are understood up to an additive error $\delta$). This follows from the definition of the point $P_0'$ as a projection and Lemma \ref{orthTriangle}.

Hence, if $t_1, t_2 \ge t_\infty$, $d(P_1, P_2) = d(P_1, P_0) + d(P_2, P_0) - 2d(P_0, P_0') = t_1 + t_2 - 2t_\infty$.

If $t_1 \le t_\infty \le t_2$, $d(P_1, P_2) = d(P_1, P_0') + d(P_0', P_2) = t_2 - t_1$.

Finally, if $t_1, t_2 \le t_\infty$, we get $d(P_1, P_2) = |t_1 - t_2| = t_1 + t_2 - 2\min\{t_1, t_2\}$ as $P_1$ lies near $P_0P_2^\infty$.

\end{proof}

\subsection{Construction of quasi-isometry}
Although the quasi-isometry which will be constructed in this section can seem to be a bit naive, it will allow us to establish an example of logarithmic quasi-isometric distortion in section \ref{unipotent}
\begin{lemma}
Let $Z$ and $Z'$ be two hyperbolic metric spaces. Let $\Theta$ be the radial extension of a boundary homeomorphism $\theta$, as described at the beginning of this section.
Then for any two points $P_1$, $P_2\in B(P_0, R) \subset Z$ such that $d(P_1,P_2) > c$, we have 
$$\frac{d_{Z'}(\Theta(P_1), \Theta(P_2))}{d_Z(P_1, P_2)} \le 1+2\frac{K(R)}{c}.$$
If $d(P_1,P_2) < c$, $$d_{Z'}(\Theta(P_1), \Theta(P_2)) < 2K(R) + c.$$
\end{lemma}

\begin{proof}
We will use the same notations as in Lemma \ref{distApprox}. Visual distance $d_Z^\infty$ between $P_1^\infty$ and $P_2^\infty$ and the (horo-)distance $t_\infty$ from $P_0$ to $P_1^\infty P_2^\infty$ are connected by the relation $e^{-t_\infty} = d_{\infty}(P_1^\infty, P_2^\infty)$. In the same way we define $t'_\infty$ as the (horo-)distance for corresponding images.

By Lemma \ref{distApprox} we know that $d(P_1,P_2) = t_1 + t_2 - 2\min\{t_1, t_2, t_\infty\}$.

Assume first $d(P_1,P_2) > c$. We will write $d_Z = d(P_1, P_2)$ for the distance between $P_1$ and $P_2$ and $d_{Z'} = d(\Theta(P_1), \Theta(P_2))$ for the distance between their images.

We have to consider four cases depending on the relative sizes of $t_1, t_2, t_0$ and $t'_\infty$ as they determine values of minima defining $d_Z$ and $d_{Z'}$. Without loss of generality, we may assume that $t_1 \le t_2$.

\emph{1st case.}
If both $t_1 < t_\infty$ and $t_1 < t'_\infty$, then
$$\frac{d_{Z'}}{d_{Z}} = \frac{t_2 - t_1}{t_2 - t_1} = 1,$$
and this case is trivial.

\emph{2nd case.}
If $t_\infty < t_1$ and $t'_\infty < t_1$. We have to give an upper bound for
$$\frac{d_{Z'}}{d_{Z}} = \frac{t_1 + t_2 - 2t'_\infty}{t_1 + t_2 - 2t_0^{\infty}}.$$
Consider
$$t'_\infty - t_\infty = \log\frac{d_\infty(\theta(P_1^\infty),\theta(P_2^\infty))}{d_\infty(P_1^\infty,P_2^\infty)}.$$
Because $d_Z > c$, we have $t_1 + t_2 - 2t_\infty > c$ hence $e^{(t_1 + t_2)/2} e^{-t_\infty} > e^{c/2}$.
And as $t_1, t_2 \le R$ we obtain for visual distance
$d^\infty_Z \ge e^{c/2}e^{-R} \ge e^{-R}.$
We conclude that
$$|t'_\infty - t_\infty| \le K(R).$$
Finally,
$$\frac{d_{Z'}}{d_Z} = \frac{d_{Z'} - d_Z + d_Z}{d_Z} = 1 + \frac{t'_\infty - t_\infty}{t_1 + t_2 - t_\infty} \le 1 + \frac{1}{c}|t'_\infty - t_\infty|.$$

\emph{3d case.}
Now let $t_\infty < t_1 < t'_\infty$. Then 
$$d_{Z'} - d_Z = t_2 - t_1 - (t_1 + t_2 - 2t_\infty) = 2(t_\infty - t_1) \le 0,$$
which leads to
$$\frac{d_{Z'}}{d_Z} \le 1.$$

\emph{4th case.}
Finally if $t'_\infty< t_1 < t_0^{\infty}$ then
$$d_{Z'} - d_Z = (t_1 + t_2 - 2t'_\infty) -(t_2 - t_1)= 2(t_1 - t'_\infty) \le 2(t_0^{\infty} - t'_\infty).$$

We know that $t_1 \le R$ and at the same time we have $t'_\infty < t_1$, hence $t'_\infty  < R$ and visual distance between $P_1^{\infty\prime}$ and $P_2^{\infty\prime}$ is at least $e^{-R}$. Now as in the 2nd case we obtain that $t_0^{\infty} - t'_\infty\le K(R)$ and hence
 
$$\frac{d_{Z'}}{d_Z}  \le 1+2\frac{K(R)}{c}.$$

Now assume that $d_Z(P_1, P_2) \le c$ (we still suppose $t_1 \le t_2$), hence the distance $t_\infty > t_2$ and we are either in first or fourth situation. In the first case, $t_1 < t_\infty$ and $t_1 < t'_\infty$ so $d_{Z'} = d_Z \le c$. In the fourth case, we have still $d_{Z'} - d_Z \le 2K(R)$ and hence $d_Z' \le c + 2K(R)$.
\end{proof}

Applying the Lemma both to $\Theta$ and $\Theta^{-1}$, we get the following Theorem.

\begin{theorem}\label{Theta}
Let $X, Y$ be two geodesic hyperbolic metric spaces with base points $x_0 \in X$, $y_0 \in Y$. Assume that there exists a constant $D$ such that for any $x \in X$ there exists a geodesic ray $\gamma$ from the base point $\gamma(0) = x_0$ and passing near $x$: $d(x, \gamma) < D$ (Hypothesis \ref{hypo1}). Let the restriction of $\Theta: \partial X  \to \partial Y$ be a homeomorphism between ideal boundaries.
Then the restriction of $\Theta$ on a ball $B(x_0,R)\subset X$ of radius $R$ is a $(\lambda, C_q)$-quasi-isometry to $B(y_0,R)\subset Y$, where $\lambda = 1+2\frac{K(R)}{c}$ and $C_q = 2K(R)+c$. The constant $c$ can be chosen as $c = D +\delta$ where $\delta$ is the hyperbolicity constant.
\end{theorem}

\section{Examples}

\subsection{Bi-H\"older maps} 

Let $\theta$ be a bi-H\"older map:
\begin{eqnarray}\nonumber
d(\theta(\xi_1), \theta(\xi_2)) \le cd(\xi_1, \xi_2)^\alpha, \alpha < 1,
\\\nonumber
d(\theta(\xi_1), \theta(\xi_2)) \ge \frac{1}{c}d(\xi_1, \xi_2)^\beta, \beta > 1.
\end{eqnarray}

Assume first that for two points $\xi_1, \xi_2$ of the ideal boundary, the visual distance $d(\xi_1, \xi_2) > e^{-R}$. Then we have
$$\log \frac{d(\theta(\xi_1), \theta(\xi_2))}{d(\xi_1, \xi_2)} \le \log cd(\xi_1, \xi_2)^{\alpha - 1} = -(1 - \alpha)\log d(\xi_1, \xi_2) \lesssim (1 - \alpha) R.$$
Now, if the visual distance between images of $\xi_1$ and $\xi_2$ satisfy $d(\theta(\xi_1), \theta(\xi_2)) > e^{-R}$, we get
$$d(\xi_1, \xi_2) \ge \frac{1}{c^{1/\alpha}}e^{-R/\alpha}$$
and hence
$$\log \frac{d(\theta(\xi_1), \theta(\xi_2))}{d(\xi_1, \xi_2)} \gtrsim \frac{1 - \alpha}{\alpha} R.$$

We obtain the lower bound for $\log \frac{d(\theta(\xi_1), \theta(\xi_2))}{d(\xi_1, \xi_2)}$ just in the same way as the upper-bound. If $d(\xi_1, \xi_2) > e^{-R}$
$$\log \frac{d(\theta(\xi_1), \theta(\xi_2))}{d(\xi_1, \xi_2)} \ge \log \frac{1}{c}d(\xi_1, \xi_2)^{\beta - 1} = - (1 - \beta)\log d(\xi_1, \xi_2) \lesssim (1 - \beta) R.$$
If $d(\theta(\xi_1), \theta(\xi_2)) > e^{-R}$
$$\log \frac{d(\theta(\xi_1), \theta(\xi_2))}{d(\xi_1, \xi_2)} \ge \log \frac{1}{c}d(\theta(\xi_1), \theta(\xi_2))^{(\beta - 1)/\beta} = - \frac{1 - \beta}{\beta}\log d(\theta(\xi_1), \theta(\xi_2)) \gtrsim \frac{1 - \beta}{\beta} R.$$
This gives
\begin{eqnarray*}
K(R)\lesssim \max\{1-\alpha,1-\beta\}R.
\end{eqnarray*}

In particular, consider two variants of the space $T^n\times[0, +\infty)$ $Z$ and $Z'$ with metrics $dt^2+\sum e^{2\mu_it}dx_i^2$ and $dt^2+\sum e^{2\mu_i't}dx_i^2$ respectively. The visual distance between points $P_1$ and $P_2$ is given by
$$d_\infty(P_1, P_2) \sim \max |x_i^1 - x_i^2|^{1/\mu_i}.$$

Pick the identity map $\theta: \partial Z \to \partial Z'$. Then
$$\frac{d_\infty(\theta(P_1), \theta(P_2))}{d_\infty(P_1, P_2)} \sim \frac{\max_i |x_i^1 - x_i^2|^{1/\mu_i'}}{\max_i|x_i^1 - x_i^2|^{1/\mu_i}} \le \max_i |x_i^1 - x_i^2|^{1/\mu_i' - 1/\mu_i}.$$
Suppose that $d(P_1, P_2) > e^{-R}$. Then
\begin{eqnarray}\nonumber
\left|\log\frac{d_\infty(\theta(P_1), \theta(P_2))}{d_\infty(P_1, P_2)}\right|\le \left|\log \max_i|x_i^1 - x_i^2|^{1/\mu_i' - 1/\mu_i}\right| =
\\\nonumber
= \max_i\left(\mu_i\left|\frac{1}{\mu_i'} - \frac{1}{\mu_i}\right| \left|\log|x_i^1 - x_i^2|^{1/\mu_i}\right|\right) \le \max_i\left|\frac{\mu_i}{\mu_i'} - 1\right| R.
\end{eqnarray}
So, we conclude that $K(R) = \left|\max_i(\mu_i/\mu_i') - 1\right| R$.

\begin{remark}
More generally, such bi-H\"older maps exist between boundaries of arbitrary simply connected Riemannian manifolds with bounded negative sectional curvature. The H\"older exponent is controlled by sectional curvature bounds.
\end{remark}

\subsection{Unipotent locally homogeneous space}\label{unipotent}

Now assume the space $Z$ is a quotient $\R^2/\Z^2\times \R$ of the space $\R^2\times \R$ with the metric $dt^2 + e^{2t}(dx^2+dy^2)$. Consider the space $Z'=\R^2/\Z^2\ltimes_{\alpha} \R$, quotient of the space $\R^2\rtimes_{\alpha} \R$, where $\alpha$ is the $2\times 2$ matrix
$$\left(\begin{array}{cc}1 & 1 \\ 0 & 1\end{array}\right).$$
The locally homogeneous metric is of the form $dt^2 + g_t$ where $g_t = (e^{t\alpha})^\ast g_0$
$$e^{t\alpha}\left(\begin{array}{c}x \\ y\end{array}\right) = \left(\begin{array}{cc} e^t & te^t \\ 0 & e^t\end{array}\right)\left(\begin{array}{c}x \\ y\end{array}\right) = \left(\begin{array}{c}e^tx + te^ty \\ e^ty\end{array}\right)$$
and so $g_t = d(e^tx+te^ty)^2 + d(e^ty)^2 = e^{2t}(dx^2 + 2tdxdy + (t^2+1)dy^2).$

Let $\theta: \partial Z \to \partial Z'$ be the identity. Consider two points $P_1 = (x_1, y_1)$ and $P_2 = (x_2, y_2)$ in $Z$. We will write $x = x_1 - x_2$ and $y = y_1 - y_2$. For the visual distance between $P_1, P_2$ we have
$$d_\infty(P_1, P_2) = \max\{|x|, |y|\}.$$
For their images $\theta(P_1)$ and $\theta(P_2)$ (see section $5$ of \cite{SXie} and \cite{Xie})
$$d_\infty(\theta(P_1), \theta(P_2)) = \max\{|y|, |x - y\log |y|\}.$$
First we will give an upper-bound for $\log(d_\infty(\theta(P_1), \theta(P_2))/d_\infty(P_1, P_2))$. We have four different cases.

\medskip

\emph{1st case.} If $|x| < |y|$ and $|x - y\log |y|| < |y|$,
$$\frac{d_\infty(\theta(P_1), \theta(P_2))}{d_\infty(P_1, P_2)} = 1.$$

\emph{2nd case.} If $|x - y\log |y|| < |y| < |x|$,
$$\frac{d_\infty(\theta(P_1), \theta(P_2))}{d_\infty(P_1, P_2)} < 1.$$

\emph{3d case.} If $|x| < |y| < |x - y\log |y||$.
$$\frac{d_\infty(\theta(P_1), \theta(P_2))}{d_\infty(P_1, P_2)} = \frac{|x - y\log y|}{|y|} \le \frac{|x|}{|y|} + |\log |y||.$$
If $d_\infty(P_1, P_2) > e^{-R}$ we have $e^{-R} < |y| \le 1$ (the upper bound follows from the fact that $y$ is a coordinate of a point of a torus) and hence $|\log |y|| \le R$ and we finish as follows,
$$\frac{d_\infty(\theta(P_1), \theta(P_2))}{d_\infty(P_1, P_2)} \le \frac{|x|}{|y|} + |\log |y|| \le 1 + R.$$
If $d_\infty(\theta(P_1), \theta(P_2)) > e^{-R}$ we will consider two situations.
\begin{itemize}
\item If $|x| > |y\log |y||$ then $|x - y\log y| < 2|x|$ and as $|x| < |y|$,
$$\frac{d_\infty(\theta(P_1), \theta(P_2))}{d_\infty(P_1, P_2)} \le 2.$$
\item If $|x| < |y\log |y||$ then $e^{-R} < |x - y\log |y|| < 2|y\log |y||$ and hence $|\log |y|| < R$, so
$$\frac{d_\infty(\theta(P_1), \theta(P_2))}{d_\infty(P_1, P_2)} \le 1 + R.$$
\end{itemize}

\emph{4th case.} Let now $|y| < |x|$ and $|y| < |x - y\log |y||$
$$\frac{d_\infty(\theta(P_1), \theta(P_2))}{d_\infty(P_1, P_2)} = \frac{|x - y\log |y||}{|x|} \le 1 + \frac{|y\log |y||}{|x|}.$$
We will check two possibilities.
\begin{itemize}
\item If $|y| \le |x|^{2}$ then
$$\frac{|y\log |y||}{|x|} = \frac{|y|^{1/2}}{|x|}\left||y|^{1/2}\log |y|\right| \le 1.$$
\item Now suppose that $|y| \ge |x|^{2}$. If $d_\infty(P_1, P_2) > e^{-R}$, we see easily that $|y| \ge e^{-2R}$ and hence
$$\frac{|y\log |y||}{|x|} \le \frac{|x \log|y||}{|x|} \le |\log |y|| \le 2R.$$ 
\end{itemize}
If $d_\infty(\theta(P_1), \theta(P_2)) > e^{-R}$ we use the fact that $|a + b| \ge 2\max\{|a|, |b|\}$. Hence, either $|x| > e^{-R}/2$ or $|y\log|y|| > e^{-R}/2$ and so $|y| \gtrsim e^{-R}$ and we finish the estimation as earlier.

So in the fourth case we have also
$$\frac{d_\infty(\theta(P_1), \theta(P_2))}{d_\infty(P_1, P_2)} \le 2R.$$

\medskip

Here, we have proved that $\log(d_\infty(\theta(P_1), \theta(P_2))/d_\infty(P_1, P_2)) \le \log R$. Now we proceed to give also a lower bound for this expression.

\emph{1st case.} If $|x| < |y|$ and $|x - y\log |y|| < |y|$,
$$\frac{d_\infty(\theta(P_1), \theta(P_2))}{d_\infty(P_1, P_2)} = 1.$$

\emph{2nd case.} If $|x - y\log |y|| < |y| < |x|$,
$$\frac{d_\infty(\theta(P_1), \theta(P_2))}{d_\infty(P_1, P_2)} = \frac{|y|}{|x|}.$$
Without loss of generality, assume $x > 0$. By the construction of $Z$, $|y| < 1$ hence $\log |y| < 0$. If $0 < x \le y\log|y|$, we have $y < 0$. Now transform $x \le y \log |y|$ as $1 \le -\log|y| (-y)/x$, hence
$$-\frac{y}{x} \ge -\frac{1}{\log |y|}.$$
Now either $d_\infty(\theta(P_1), \theta(P_2)) = |y| > e^{-R}$ or $e^{-R} \le d_\infty(P_1, P_2) = |x| \le y\log|y|$ which also means that $|y| \gtrsim e^{-R}$. So,
$$\frac{|y|}{|x|} \ge \frac{1}{R}.$$

If on the contrary $y\log|y| \le x$ we have
\begin{equation}\label{case2}
x - y\log|y| < |y| < x.
\end{equation}
First we notice that $y\log|y| > x - |y| > 0$. As $|y| < 1$ for any point of our space, $\log |y| < 0$ and we conclude that $y < 0$. Now from (\ref{case2}) we obtain that $x < -y(1 - \log|y|)$. As $1 - \log|y| > 0$ we obtain
$$-\frac{y}{x} > \frac{1}{1 - \log|y|}.$$
If $d_\infty(\theta(P_1), \theta(P_2)) = |y| > e^{-R}$, we trivially get that
$$\frac{|y|}{|x|} > \frac{1}{R}.$$
If $e^{-R} \le d_\infty(P_1, P_2) = |x|$ we write $e^{-R} < x < -y(1 - \log|y|)$ and hence $y \gtrsim e^{-R}$, so we obtain the same result. So, in both cases we come to the same result
$$\left|\log\frac{|y|}{|x|}\right| < R.$$

\emph{3d case.} Assume $|x| < |y| < |x - y\log |y||$, this case is trivial as
$$\frac{d_\infty(\theta(P_1), \theta(P_2))}{d_\infty(P_1, P_2)} = \frac{|x - y\log y|}{|y|} \ge 1.$$

\emph{4th case.} Let now $|y| < |x|$ and $|y| < |x - y\log |y||$. We also suppose that $x > 0$ to save notation.
\begin{equation}\label{case4}
\frac{d_\infty(\theta(P_1), \theta(P_2))}{d_\infty(P_1, P_2)} = \frac{|x - y\log |y||}{|x|} = \left|1 - \frac{y\log|y|}{x}\right|.
\end{equation}
If (\ref{case4}) is greater than $1/2$ then we have nothing to prove. So suppose that (\ref{case4}) is less than $1/2$
$$-\frac{x}{2} \le x - y\log|y| \le \frac{x}{2},$$
and so
$$\frac{x}{2} \le y\log|y| \le \frac{3x}{2}.$$
The last inequality shows that if either $d_\infty(\theta(P_1), \theta(P_2)) \ge e^{-R}$ or $d_\infty(P_1, P_2) \ge e^{-R}$, $|y| \gtrsim e^{-R}$ and so we have 
$$\frac{|y\log|y||}{x} \ge \frac{|y\log|y||}{y} = |\log|y|| \ge \frac{1}{R},$$
which completes our discussion of this example. We have proved that
$$K(R) \lesssim \log R.$$

\section{Appendix: Quasi-isometric embeddings and fundamental groups}

Here we would like to discuss the hypothesis of the Theorem \ref{lowerBound} that the quasi-isometric embedding under consideration is a homotopy equivalence. We will show that if $\mathrm{dim}(Z)\geq 3$, one may believe that the assumption that $\Theta$ be isomorphic on fundamental groups is not that restrictive. Indeed, in Proposition \ref{pi1}, we shall show that this is automatic, but unfortunately the argument introduces an ineffective constant $R_0$, which makes it useless. For instance, if it turns out that $R_0=\lambda_1^2$, Proposition \ref{pi1} does not help to remove the homotopy assumption in Theorem \ref{lowerBound}. Nevertheless, it is included for completeness sake.

\begin{proposition}\label{pi1}
Let $Z, Z'$ be two spaces of the described form with equal dimensions $n+1 \ge 3$. Then for any $\lambda_1\ge1,\lambda_2\ge1,c_1\ge0,c_2\ge0$ there exists $R_0=R_0(\lambda_1,\lambda_2,c_1,c_2)$ such that if $R > R_0$ and a continuous map $f:B_{Z_\mu}(O,R_0)\to Z_{\mu'}$ is a $(\lambda_1,\lambda_2,c_1,c_2)$-quasi-isometric embedding, then $f$ induces an isomorphism on fundamental groups $\pi_1(Z_\mu)\to\pi_1(Z_{\mu'})$.
\end{proposition}

\begin{proof}
We provide a proof by contradiction. Assume that for arbitrarily large values of $R$, there exists a map $f_R: B_{Z}(R)\to Z'$ which is a $(\lambda_1,\lambda_2,c_1,c_2)$-quasi-isometric embedding which is not isomorphic on fundamental groups. Pick 
 a $2c_1/\lambda_1$-dense and $c_1/\lambda_1$-discrete subset $\Lambda$ of $Z$. Notice that if $f_R$ is a $(\lambda_1,\lambda_2,c_1,c_2)$-quasi-isometry, then $f_R$ is bi-Lipschitz on $B_Z(R)\cap\Lambda$. Conversely, if a map defined on $B(R)\cap\Lambda$ is bi-Lipschitz, then it can be continuously extended on $B(R)$ as a quasi-isometric embedding. Indeed, away from a ball, $Z'$ is contractible up to scale $c_1$.

Set $\rho = d(O', f_R(O))$. First, consider the case when $\rho \to \infty$. Set $\sigma = (\rho/4-c_1)/\lambda_1$. Then $f_R(B(O,\sigma))$ is contained in a ball $B(f_R(O),\rho/4)$ which lies in the complement of $B(O',\rho/2)$
$$f_R(B(O,\sigma)) \subset B(f_R(O),\rho/4) \subset B(O',\rho/2)^c.$$
The diameter of the image of any loop in $B(O,\sigma)$ is at most $\lambda_1\sigma+c_1$. Because $\lambda_1\sigma+c_1 < \rho/4$, these loops are homotopic to $0$ (diameters of loops are too short relatively to $B(O',\rho/2)^c$). Hence, the restriction of $f_R$ on $B(0,\sigma)$ is homotopic to $0$. Hence $f_R$ lifts to $\tilde f_R: B_{Z}(\sigma) \to \tilde Z'=X_{\mu'}$ which is homogeneous. Now up to composing $\tilde f_R$ with an isometry we can suppose that it preserves the center $\tilde f_R(O) = O'$. By Ascoli's theorem, we can find a sequence $\tilde f_{R_j}|_\Lambda$ which uniformly converges to $\tilde f|\Lambda: Z\cap\Lambda\to \tilde Z'$ which is also bi-Lipschitz. We continuously extend $\tilde f_{|\Lambda}$ to $\tilde f: Z\to \tilde Z'$, $\tilde f$ is a quasi-isometric embedding. Its extension to ideal boundaries is continuous and injective. By the theorem of invariance of domain, $\partial\tilde f: T^{n} \simeq \partial X_\mu = S^n$ is open, and thus a homeomorphism. This provides a contradiction if $n\ge 2$.

If $\rho=d(O',f_R(O))$ stays bounded, we can directly use Ascoli's theorem, and get a limiting continuous quasi-isometric embedding $f$. Again, $f$ extends to the ideal boundary, $\partial f:\partial Z\to \partial Z'$, the map $\partial f$ is continuous and injective. Because $\partial Z$ and $\partial Z'$ have the same dimension, $\partial f$ is an open map by the theorem of invariance of domain and $\partial f$ is a homeomorphism. Hence, $\partial f$ induces an isomorphism on fundamental groups. If $R_j$ is sufficiently large, then $f_{R_j}$ is at bounded distance from $f$ and hence $f_{R_j}$ also induces an isomorphism $\pi_1(B_{Z}(R))\to\pi_1(Z')$. This contradiction completes the proof.
\end{proof}

\begin{remark}
The proof does not provide an effective value of $R_0$.
\end{remark}

\section{Ackowledgement}

The author thanks Pierre Pansu for his invaluable help through all steps of this work on the paper. This work is supported by Agence Nationale de la Recherche, grant ANR-10-BLAN 0116. This work is also partially supported by RFBR, grant 14-07-00812.

\end{document}